\documentclass[12pt]{amsart}
\usepackage{amssymb,amsmath}
\usepackage{showkeys}

\def\({\left(}
\def\){\right)}
\def\Nx{\nabla_x}
\def\Cal{\mathcal}
\def\Om{\Omega}
\def\eb{\varepsilon}
\def\dist{{\rm dist}}

\def\R {\mathbb{R}}

\newcommand{\be}{\begin{equation} }
\newcommand{\ee}{\end{equation} }
\def\dd{\,d}

\def \and{\qquad\text{and}\qquad}

\def\Bbb{\mathbb}
\def\Dt{\partial_t}
\def\Dx{\Delta_x}
\def\R{\Bbb R}
\def\meas{\operatorname{mes}}

\newtheorem{proposition}{Proposition}[section]
\newtheorem{theorem}[proposition]{Theorem}
\newtheorem{corollary}[proposition]{Corollary}
\newtheorem{lemma}[proposition]{Lemma}
\theoremstyle{definition}
\newtheorem{definition}[proposition]{Definition}
\newtheorem{remark}[proposition]{Remark}

\numberwithin{equation}{section}
\catcode`\ç\active\def ç{{\c c}} \catcode`\Ç\active\def Ç{{\c C}}
\catcode`\ð\active\def ð{{\u g}} \catcode`\Ð\active\def Ð{{\u G}}
\catcode`\ý\active\def ý{{\i}} \catcode`\Ý\active\def Ý{{\. I}}
\catcode`\ö\active\def ö{{\" o}} \catcode`\Ö\active\def Ö{{\" O}}
\catcode`\þ\active\def þ{{\c s}} \catcode`\Þ\active\def Þ{{\c S}}
\catcode`\ü\active\def ü{{\" u}} \catcode`\Ü\active\def Ü{{\" U}}

\def \no#1#2#3 {{\bf #1} (#3), #2.}
\def \eds#1#2#3 {#1, #2, #3.}

\title[Cahn-Hilliard equations in cylindrical domains]
{Infinite energy solutions for the Cahn-Hilliard equation in cylindrical domains}

\author[]
{A. Eden, V. K. Kalantarov and S. V. Zelik}

\thanks{ The work is  supported by TUBITAK, ISBAB project No:107T896.
The Authors would like to thank Giulio  Schimperna for fruitful
discussions.}
\address{Department of mathematics, Boðaziçi University,
 \newline\indent
Bebek, Istanbul, Turkey} \email{eden@boun.edu.tr}

\address{Department of mathematics, Ko{\c c} University,
\newline\indent Rumelifeneri Yolu, Sariyer 34450\newline\indent
Sariyer, Istanbul, Turkey} \email{vkalantarov@ku.edu.tr}
\address{Department of mathematics,
\newline\indent
University of Surrey Guildford,
\newline\indent  GU2 7XH United Kingdom }
\email{S.Zelik@surrey.ac.uk}

 \keywords{}

\begin{document}

\begin{abstract}
We give a detailed study of the infinite-energy solutions of the Cahn-Hilliard equation in the 3D cylindrical domains
in uniformly local phase space. In particular, we establish the well-posedness and dissipativity for the case of regular potentials of arbitrary polynomial growth as well as for the case of sufficiently strong singular potentials. For these cases, we prove the further regularity of solutions and the existence of a global attractor. For the cases where we have failed to prove the uniqueness (e.g., for the logarithmic potentials), we establish the existence of the trajectory attractor and study its properties.
\end{abstract}

\maketitle
\section{Introduction}
We study the classical Cahn-Hilliard equations
\begin{equation}\label{0.ch}
\Dt u=\Dx(-\Dx u-f(u)+g)
\end{equation}
considered in an {\it unbounded} cylindrical domain $\Omega=\R\times\omega$ ($\omega$ is a smooth bounded domain) of $\R^3$ endowed by the Dirichlet boundary conditions.
\par
As known, the Cahn-Hillard equation is central for the material sciences and extensive amount of papers are devoted to the mathematical analysis of this equation and various of its generalizations. In particular, in the case where $\Omega$ is bounded,
its analytic and dynamic properties are relatively well-understood including the well-posedness (even in the case of singular potentials $f$) and dissipativity, smoothness, existence of global and exponential attractors, upper and lower bounds for the dimension, etc. We mention here only some contributors, namely, \cite{CH,D,EK,EGZ,EMZ1,EGW,GSZ,GSZ1,GPS,Ka,EK1,MZ3,MZ2,MZ1,No1,No,WW} (see also the references therein).
\par
The situation in the case where the underlying domain is unbounded is essentially less clear even in the case of finite-energy solutions. Indeed, as well-known, the key feature of the Cahn-Hilliard (CH) equation in bounded domains which allows to build up a reasonable theory (especially in the case of rapidly growing or singular nonlinearities) is the possibility to obtain good estimates in the negative Sobolev space $H^{-1}(\Omega)$ and, to this end, one should use the inverse Laplacian $(-\Dx)^{-1}$.
But, unfortunately this operator is not good in unbounded domains (in particular, does not map $L^2(\R^3)$ to $L^2(\R^3)$) and this makes the most part of analytic tools earlier developed for the Cahn-Hilliard equation unapplicable to the case of unbounded domains. Thus, despite the general theory of dissipative PDEs in unbounded domains which seems highly developed now-a-days (see the surveys \cite{MZ} and \cite{Ba} and references therein), even the dissipativity of finite-energy solutions in $\Omega=\R^n$ is not known for the CH equations under the reasonable assumptions on the non-linearity $f$ (to the best of our knowledge, it is established only if $f$ is linear outside of the large ball in $\Omega$, see \cite{CM}).
\par
This problem partially disappears if we consider the case where $\Omega$ is cylindrical domain endowed by the Dirichlet boundary conditions (which is the main topic of the present paper). In that case, the inverse Laplacian is well-defined (similar to the case of bounded domains) and the theory of {\it finite-energy} solutions can be built straightforwardly combining the usual Cahn-Hilliard technique and the weighted technique  (see \cite{Ab1,Ab2,BV,MZ,EZ,EMZ,Z1}).
\par
However, as also well-known, the class of {\it finite-energy} solutions is not satisfactory in unbounded domains (e.g., it does not contain physically important solutions, like spatially-periodic patterns and requires the additional strong restrictions on the potential $f$ and external forces $g$) and should be naturally replaced by the solutions in the so-called uniformly-local Sobolev spaces which typically have {\it infinite-energy}, see e.g. the survey \cite{MZ} for further discussion.
\par
Thus, following the general strategy, it seems natural to consider the Cahn-Hilliard equation in the uniformly-local phase spaces and, in order to obtain the reasonable estimates, we need to use the weighted energy estimates. But, unfortunately, the application of that technique to the Cahn-Hilliard equation is far from being straightforward even in the case of cylindrical domains since the presence of the weight destroys the  $H^{-1}$-estimates. For that reason, the well-posedness of the Cahn-Hilliard equation in uniformly local spaces was known before only for the regular potentials with strong growth restrictions and only under the presence of the regularizing terms (the so-called microforces, see \cite{Bo}) where the $H^{-1}$-estimates are  not necessary.
\par
The aim of the present paper is to give a systematic study of the Cahn-Hilliard equations in cylindrical domains in the uniformly-local phase spaces. To this end, we adapt some technique initially invented for the Navier-Stokes equations in cylinders, see \cite{Z2} and \cite{Z3} which allow us to restore the crucial $H^{-1}$-estimates for the reasonable classes of regular and singular potentials and to verify the well-posedness, dissipativity and existence of global attractors for these
potentials. In particular, we are able to treat the regular potentials of arbitrary polynomial growth (of course, under the standard dissipativity assumptions) as well as some classes of singular potentials. For instance, we prove the existence and uniqueness for the nonlinearities like
$$
f(u)\sim\frac{u}{(1-u^2)^\gamma}-Ku
$$
with $\gamma\ge5/3$. Unfortunately, we are unable to verify the uniqueness for the most physical logarithmic potentials
$$
f(u)\sim\log\frac{1+u}{1-u}-Ku
$$
and for this reason, we  will only construct below the trajectory attractor for the associated Cahn-Hilliard equation.
\par
The paper is organized as follows. In Section \ref{s1}, we briefly recall the formalism of weighted energy estimates and prove the
dissipative estimates in the appropriate weighted and uniformly local Sobolev space. The central Section \ref{s2} is devoted to the uniqueness problem for the Cahn-Hilliard equations in uniformly local spaces. We also establish here the smoothing property and the separation of the solutions from the singular points of the potential $f$.  Then, in Section \ref{s3}, we study the attractors. We start with the cases where the uniqueness is verified and prove the existence of a "usual" uniformly-local attractor. After that, we turn to the case without uniqueness and verify the existence of the so-called trajectory attractor in the weak topology of the trajectory phase space. In Section \ref{s4}, we adopt  the general method presented in \cite{MiZe} to the Cahn-Hilliard problem and verify that any weak solution satisfies the weighted energy {\it equality}. Finally, based on that equality, we extend the so-called  energy method (see \cite{Ball,MRW}, see also \cite{CVZ1,CVZ2}) to the case of uniformly local phase spaces and trajectory attractors and deduce the compactness of the attractor in a strong topology as well as the attraction to it in that strong topology.
\par
Finally, in Section \ref{s5}, we note that, although we consider here only cylindrical domains $\Omega$, most part of our results can be straightforwardly extended to any unbounded domain which possesses the Friedrichs inequality, in particular, for a domain in space between two parallel planes ($\Omega=\R^2\times(0,1)$).

\section{A priori estimates and the existence of solutions}\label{s1}
In this section, we consider the following initial boundary value problem for the Cahn-Hilliard equation in a cylindrical
domain $(x_1,x_2,x_3)\in\Omega=\R\times\omega$ where $\omega$ is a smooth bounded domain of $\R^2$:
\begin{equation}\label{eqmain}
\begin{cases}
\Dt u=\Dx\mu,\ \ \mu:=-\Dx u+f(u)+g,\\
u\big|_{t=0}=u_0,\ \ u\big|_{\partial\Omega}=\mu\big|_{\partial\Omega}=0,
\end{cases}
\end{equation}
where $u=u(t,x)$ and $\mu=\mu(t,x)$ are the unknown order parameter and the chemical potential respectively, $g$ is a given external force and $f$ is a given nonlinearity.
\par
Since we do not impose any decay conditions on the solutions as $x\to\infty$, it is natural (following to the general theory
of dissipative PDEs in unbounded domains, see \cite{EZ,MZ,Z1,Z4} and references therein) to consider equation \eqref{eqmain} in the so-called uniformly local Sobolev spaces $W^{l,p}_b(\Omega)$. We recall that these spaces can be defined as a subspace of $W^{l,p}_{loc}(\overline\Omega)$ on which the  following norm is finite:
\begin{equation}\label{ubnorm}
\|u\|_{W^{l,p}_b}:=\sup_{s\in\R}\|u\|_{W^{l,p}(\Omega_{[s,s+1]})},\ \ \Omega_{[s,s+1]}:=[s,s+1]\times\omega,
\end{equation}
see \cite{Z4} for the details. Remind also that $L^p_{loc}(\overline{\Omega})$ stands for the Frechet space generated by the seminorms $\|u\|_{L^p(\Omega_{[s,s+1]})}$, $s\in\R$, and the spaces $W^{l,p}_{loc}(\overline{\Omega})$ are defined analogously.
 In addition, will  use the spaces $L^p_b(\R_+\times\Omega)$ of functions depending on space and time variables and the corresponding Sobolev spaces which can be defined in a similar way.
\par
We assume that the external force $g\in L^2_b(\Omega)$ and the initial data $u_0\in \Phi_b$ with
\begin{equation}\label{bphase}
\Phi_b:=\{u_0\in W^{1,2}_b(\Omega),\ F(u_0)\in L^1_b(\Omega)\},\ \ F(u):=\int_0^uf(v)\,dv.
\end{equation}
We are now able to define a solution of the CH problem \eqref{eqmain}
\begin{definition}\label{defsol} A  function $u$, is a (weak, infinite-energy) solution of  problem \eqref{eqmain} if
\begin{equation}\label{eqdef}
\begin{cases}
1.\ \ u\in L^\infty(\R_+,\Phi_b)\cap C([0,\infty),L^2_{loc}(\overline\Omega)),\\
2.\ \ f(u),\Dx u,\Nx\mu\in L^2_b(\R_+\times\Omega)
\end{cases}
\end{equation}
 and the equations \eqref{eqmain} are satisfied in the sense of distributions.
\end{definition}
Finally, we will consider both cases of regular and singular nonlinearities $f$ (the difference between these cases is not essential for this section, however, it will become crucial in the next sections). Namely, for the {\it regular} case, we assume that
\begin{equation}\label{freg}
\begin{cases}
1.\ \ f\in C^1(\R,\R),\\
2. \ \ f(u).u\ge -C,\ \forall u\in\R,\\
3.\ \ f'(u)\ge - K,\ \ \forall u\in\R,
\end{cases}
\end{equation}
for some positive numbers $C$ and $K$.
\par
For the singular case, we assume that $f$ is defined on the interval $u\in(-1,1)$ and
\begin{equation}\label{fsing}
\begin{cases}
1.\ \ f\in C^1(-1,1),\\
2. \ \ \lim_{u\to\pm1}f(u)=\pm\infty,\\
3.\ \ \lim_{u\to\pm1}f'(u)=+\infty
\end{cases}
\end{equation}
(which coincides with the conditions of \cite{MZ1} for the case of bounded domains). Of course, in that case, we should additionally assume that
\begin{equation}\label{less1}
-1<u(t,x)<1,\ \text{for almost all $(t,x)$}
\end{equation}
in order to make sense of the terms $f(u)$ and $F(u)$. In the sequel, we will also need the function
\begin{equation}\label{f0}
f_0(u)=f(u)+Ku
\end{equation}
which has the same behavior as $u\to\pm\infty$ in the regular case (and as $u\to\pm1$ in the singular case),
but is monotone increasing.
\par
The main result of this section is the following theorem.
\begin{theorem}\label{Th1.main} Let the nonlinearity $f$ satisfy assumption \eqref{freg} or \eqref{fsing} and the external force $g\in L^2_b(\Omega)$. Then, for any $u_0\in\Phi_b$, there exists at least one solution $u(t)$ of problem \eqref{eqmain} (in the sense of Definition \ref{defsol}) which satisfies the following dissipative estimate:
\begin{multline}\label{1.dis}
\|u(T)\|_{W^{1,2}_b(\Omega)}^2+\|F(u(T))\|_{L^1_b(\Omega)}+\|\Dx u\|_{L^2_b([T,T+1]\times\Omega)}^2+\\+\|f(u)\|^2_{L^2_b([T,T+1]\times\Omega)}+\|\Nx\mu\|^2_{L^2_b([T,T+1]\times\Omega)}\le C\|u_0\|_{\Phi_b}^2e^{-\alpha T}+C(\|g\|^2_{L^2_b(\Omega)}+1),
\end{multline}
where the positive constants $C$ and $\alpha$ are independent of $T$ and $u_0$.
\end{theorem}
\begin{proof} We first give the formal derivation of the dissipative estimate \eqref{1.dis}
 (some explanations on how to justify it and to verify the existence of a solution will be given afterwards). To this end, we multiply equation \eqref{eqmain} by $\phi(x_1) \mu$, where the weight function
$$
\phi(x_1)=\phi_{\eb,s}(x_1):=e^{-\eb\sqrt{|x_1-s|^2+1}},
$$
$s\in\R$ is arbitrary, $x:=(x_1,x_2,x_3)\in\Omega$, $\eb>0$ is a sufficiently small parameter which will be specified later. Then, after the straightforward transformations, we arrive at
\begin{multline}\label{1.est}
\frac d{dt}\(1/2\|\Nx u(t)\|_{L^2_\phi}^2+(F(u(t)),\phi)+(u\phi,g)\)+\|\Nx\mu(t)\|^2_{L^2_\phi}+\\+(\Dt u(t),\phi'\partial_{x_1}u(t))+(\partial_{x_1}\mu (t),\phi'\mu(t))=0,
\end{multline}
where we denote by
$$
\|v\|_{L^p_\phi}^p:=\int_\Omega|v(x)|^p\phi(x)\,dx
$$
the usual norm in the weighted Lebesgue space $L^p_\phi$.
\par
Using  the Poincare inequality together with the obvious inequality
\begin{equation}\label{1.phi}
|\phi'(x_1)|+|\phi''(x_1)|\le C\eb\phi(x_1),
\end{equation}
(where the constant $C$ is independent of $\eb$ and $s$),
the last term in the left-hand side of \eqref{1.est} can be estimated as follows:
\begin{equation}\label{1.est1}
2|(\partial_{x_1}\mu,\phi'\mu)|=|(\phi'',\mu^2)|\le C\eb\|\mu\|_{L^2_\phi}^2\le C'\eb\|\Nx\mu\|^2_{L^2_\phi}.
\end{equation}
 Furthermore, from equation \eqref{eqmain}, we conclude that
\begin{equation}\label{1.est2}
|(\Dt u,\phi'\partial_{x_1} u)|=(\Dx\mu,\phi'\partial_{x_1}u)\le C\eb\(\|\Nx\mu\|^2_{L^2_\phi}+\|u\|^2_{H^2_\phi}\).
\end{equation}
In order to close the estimate, we only need to use the weighted maximal regularity estimate for the elliptic equation
\begin{equation}\label{1.ell}
\Dx u-f(u)=h,\ \ u\big|_{\partial\Omega}=0
\end{equation}
\begin{lemma}\label{Lem1.ellreg} Let $u\in W^{1,2}_b(\Omega)$ be a solution of equation \eqref{1.ell} with $h\in L^2_b(\Omega)$ and the nonlinearity $f$ satisfy assumptions \eqref{freg} or \eqref{fsing}. Then,
$u\in W^{2,2}_b(\Omega)$, and the following estimate holds:
\begin{equation}\label{1.maxreg}
\|u\|_{W^{2,2}_\phi}^2+\|f(u)\|_{L^2_\phi}^2+(\phi|f(u)|,|u|)\le C\|h\|^2_{L^2_\phi}+C_\eb,
\end{equation}
where $\phi=\phi_{\eb,s}(x_1)$, the constant $C$ is independent of $\eb>0$ being small enough and the constant $C_\eb$ depends on $\eb$ (no matter singular or regular the nonlinearity $f$ is).
\end{lemma}
\begin{proof} Although the regularity estimate \eqref{1.maxreg} for the semilinear equation \eqref{1.ell} is well-known
(see, e.g., \cite{EZ}), for the convenience of the reader, we give its derivation below. To this end, we first multiply equation \eqref{1.ell} by $\phi u$. Then, using that in both cases the inequality
$f(u).u\ge -C$ holds, after the standard transformations involving \eqref{1.phi} and the Poincare inequality, we arrive at
\begin{equation}\label{1.h1}
\|\Nx u\|_{L^2_\phi}^2+(\phi |f(u)|,|u|)\le C\|h\|_{L^2_\phi}^2+C_\eb.
\end{equation}
After that, we may multiply equation \eqref{1.ell} by $\Nx(\phi\Nx u)$ and use that $f'(u)\ge-K$ which gives again after the standard transformations that
\begin{equation}\label{1.h2}
\|\Dx u\|_{L^2_\phi}^2\le (C+K)\|u\|_{W^{1,2}_\phi}^2+\|h\|^2_{L^2_\phi}.
\end{equation}
 This estimate together with \eqref{1.h1} and the weighted $L^2\to H^2$-regularity estimate for the Laplacian  give the desired estimate \eqref{1.maxreg} and finish the proof of the lemma.
\end{proof}
Applying now Lemma \ref{Lem1.ellreg} to the equation
\begin{equation}\label{1.muell}
\Dx u-f(u)=g-\mu
\end{equation}
and using now estimates \eqref{1.est1}, \eqref{1.est2} and \eqref{1.maxreg} and fixing $\eb$ being small enough, we deduce from \eqref{1.est} that
\begin{equation}\label{1.nondis}
\frac d{dt}\(\frac12\|\Nx u(t)\|_{L^2_\phi}^2+(F(u(t)),\phi)+(\phi u(t),g)\)+\|\Nx\mu(t)\|^2_{L^2_\phi}\le C(1+\|g\|^2_{L^2_\phi}),
\end{equation}
where $\eb$ is now fixed and the constant $C$ is independent of the parameter $s$ in the definition of the weight $\phi$.
\par
Finally, using Lemma \ref{Lem1.ellreg} together with the fact that $F(u)\le |f(u)|\cdot|u|+C|u|^2$, we conclude that
$$
\|\Nx u(t)\|_{L^2_\phi}^2+(F(u(t)),\phi)\le C\(\|\Nx\mu(t)\|^2+\|g\|^2_{L^2_\phi}+1\)
$$
and, therefore, inequality \eqref{1.nondis} can be rewritten in the form
\begin{multline}\label{1.ddis}
\frac d{dt}\(1/2\|\Nx u(t)\|_{L^2_\phi}^2+(F(u(t)),\phi)+(u\phi,g)\)+\\+\alpha\(1/2\|u(t)\|_{L^2_\phi}^2+(F(u(t)),\phi)+(u\phi,g)\)+\\+
\kappa\(\|\Nx\mu(t)\|^2_{L^2_\phi}+\|u(t)\|^2_{W^{2,2}_\phi}+\|f(u(t))\|^2_{L^2_\phi}\)\le C(1+\|g\|^2_{L^2_\phi}),
\end{multline}
for some positive constants $\alpha$ and $C$ which are independent of $u$ and $t$. Applying now the Gronwall inequality
to \eqref{1.dis}, we obtain the weighted analogue of the desired dissipative estimate \eqref{1.dis}:
\begin{multline}\label{1.w-dis}
\|u(T)\|_{W^{1,2}_\phi(\Omega)}^2+\|F(u(T))\|_{L^1_\phi(\Omega)}+\|\Dx u\|_{L^2_\phi([T,T+1]\times\Omega)}^2+\|f(u)\|^2_{L^2_\phi([T,T+1]\times\Omega)}+\\+\|\Nx\mu\|^2_{L^2_\phi([T,T+1]\times\Omega)}\le C\(\|u_0\|_{W^{1,2}_\phi}^2+\|F(u_0)\|_{L^1_{\phi}}\)e^{-\alpha T}+C(\|g\|^2_{L^2_\phi(\Omega)}+1)
\end{multline}
for some positive constants $\alpha$ and $C$. In addition, these constants are independent of the parameter $s$ in the weight function $\phi=\phi_{\eb,s}$. Therefore, taking the supremum on $s\in\R$ from the both sides of inequality \eqref{1.w-dis} and using that
$$
\|v\|_{L^p_b}\sim \sup_{s\in\R}\|v\|_{L^p_{\phi_{\eb,s}}},
$$
see \cite{Z4}, we obtain the desired dissipative estimate \eqref{1.dis} in the uniformly local spaces.
\par
Thus, it only remains to verify the existence of a solution. In a fact, it can be done in many standard ways. In particular, one of the simplest ways is to approximate the initial data $u_0$ and the external force $g$ by a sequence
$u_0^n$ and $g^n$ of smooth functions with finite support. Then, for the approximate Cahn-Hilliard problems \eqref{eqmain} with that data, the usual unweighted theory is applicable and the existence and uniqueness of a solution $u_n(t)$ can be verified exactly as in the case of bounded domains (without loss of generality, we may assume that $f(0)=0$). Thus, arguing as before, we obtain the dissipative estimate \eqref{1.dis} for the approximate solutions $u_n(t)$ with constants $\alpha$ and $C$ independent of $n$. Passing after that to the limit $n\to\infty$, we end up with the desired solution and the desired dissipative estimate \eqref{1.dis} in the uniformly local spaces. Since all of that arguments are standard and straightforward, we leave these details to the reader.
\end{proof}
We now formulate one more regularity result which will be useful for verifying the uniqueness of solution in the case of singular potentials.
\begin{corollary}\label{Cor1.reg26} Let the assumptions of Theorem \ref{Th1.main} hold, $g\in L^6_b(\Omega)$, and let $u$ be a solution of problem \eqref{eqmain} constructed in that theorem. Then, the following estimate holds:
\begin{equation}\label{1.reg26}
\sup_{s\in\R}\|f(u)\|_{L^2([t,t+1],L^6(\Omega_{[s,s+1]}))}\le C\|u_0\|_{\Phi_b}e^{-\alpha t}+C(1+\|g\|_{L^2_b}),
\end{equation}
where the positive constants $C$ and $\alpha$ are independent of $t$ and $u$.
\end{corollary}
\begin{proof} Indeed, without loss of generality, we may assume that $f'(u)\ge0$. Then, multiplying equation \eqref{1.muell} by $\phi f(u)^5$, we arrive at
\begin{multline}\label{1.l6}
5(\phi f'(u)f(u)^4\Nx u,\Nx u)+\|f(u)\|_{L^6_\phi}^6\le (g-\mu,\phi f(u)^5)-\\-(\phi'f(u)^5,\partial_{x_1}u)\le 1/2\| f(u)\|_{L^6_\phi}^6+C\(\|\Nx u\|_{L^6_\phi}^6+\|g\|_{L^6_\phi}^6+\|\mu\|_{L^6_\phi}^6\).
\end{multline}
Therefore,
$$
\|f(u)\|_{L^6_\phi}^2\le 2C\(\|g\|_{L^6_\phi}^2+\|\mu\|_{L^6_\phi}^2+\|\Nx u\|^2_{L^6_\varphi}\).
$$
Integrating this formula with respect to $[t,t+1]$, taking the supremum over $s\in\R$ and using that $W^{1,2}\subset L^6$
and Lemma \ref{Lem1.ellreg},
we finally have
$$
\sup_{s\in\R}\|f(u)\|_{L^2([t,t+1],L^6(\Omega_{[s,s+1]}))}\le C\(\|g\|_{L^6_b}+\|\Nx\mu\|_{L^2_b([t,t+1]\times\Omega)}\)
$$
which together with the dissipative estimate \eqref{1.dis} give \eqref{1.reg26} and finish the proof of the lemma.
\end{proof}
We recall that, up to the moment, we consider only the solutions of equation \eqref{eqmain} with sufficiently regular initial data $u_0\in\Phi_b$. However, it is well-known that, at least in the case of bounded domains, the Cahn-Hilliard equation is well-posed for less regular initial data $u_0\in H^{-1}(\Omega)$ and that the $H^{-1}$-estimates of solutions are crucial for the theory of that equation, see \cite{Ka,MZ1,Te}. In the case of unbounded domains and uniformly local phase spaces, the situation becomes more complicated and more delicate although, as we will see below, it is still possible to verify the existence of solutions for $u_0\in W^{-1,2}_b(\Omega)$ for the general case of regular and singular potentials.
\begin{theorem}\label{Th1.hm1} Let the nonlinearity $f$ satisfy conditions \eqref{freg} (regular case) together with the polynomial growth restriction
\begin{equation}\label{1.pol}
f(u).u\ge -C+C_1|u|^{p+1},\ \ |f(u)|\le C_2(1+|u|^p)
\end{equation}
for some positive constants $C$, $C_1$, $C_2$ and some exponent $p>0$,
  or \eqref{fsing} (singular case) and $u_0\in W^{-1,2}_b(\Omega)$ (in the case of singular potentials, we should assume, in addition, that $|u_0(x)|<1$ almost everywhere). Then, there exists at least one solution $u(t)$ of problem \eqref{eqmain} whichbecomes more regular ($u(t)\in\Phi_b$) for all $t>0$ and satisfy all assumptions of Definition \ref{defsol} for $t>0$. Moreover,
\begin{multline}\label{1.west}
\|u(t)\|_{W^{-1,2}_b(\Omega)}^2+\|\Nx u\|_{L^2_b([t,t+1]\times\Omega)}^2+\\+\|f(u)u\|_{L^1_b([t,t+1]\times\Omega)}\le
C\|u_0\|_{W^{-1,2}_b(\Omega)}e^{-\alpha t}+C(1+\|g\|_{L^2_b}^2),
\end{multline}
where the positive constants $C$ and $\alpha$  are independent of $u$ and $t$. In addition, the following smoothing estimate holds:
\begin{equation}\label{1.wreg}
\|u(t)\|_{\Phi_b}\le Ct^{-1/2}\|u_0\|_{W^{-1,2}_b}+C(1+\|g\|_{L^2_b}),\ \ t\le1
 \end{equation}
for some positive constant $C$ and estimate \eqref{1.dis} holds for $t>0$ with $u_0$ replaced by $u(t)$.
\end{theorem}
\begin{proof}  Let $v(t):=(-\Dx)^{-1}u(t)$, and let us rewrite problem \eqref{eqmain} in the equivalent form:
\begin{equation}\label{ch1}
\begin{cases}
\Dt v=\Dx u-f(u)+g,\ \  x\in \Om,\ t>0,\\
u(0)=u_0, \ u\big|_{\partial\Omega}=0.
\end{cases}
\end{equation}
Multiplying  this equation by $\Nx(\phi_\eb(x)\Nx v)$ and integrating over $x\in \Omega$, we get
\begin{multline}\label{1.minus}
\frac12\frac d{dt}\|\Nx v(t)\|^2_{L^2_\phi}+\|\Nx u\|^2_{L^2_\phi}+(\phi f(u),u)=\\=-(\partial_{x_1} u,\phi' \partial_{x_1} v)-(f(u),\phi'\partial_{x_1}v)+(g,\Nx(\phi\Nx v))
\end{multline}
We now use inequality \eqref{1.phi} together with the maximal regularity of the Laplacian in weighted spaces
\begin{equation}\label{1.lreg}
C_2\|\Nx u\|_{L^2_{\phi_\eb}}\le \|v\|_{H^3_{\phi_\eb}}\le C_1\|\Nx u\|_{L^2_{\phi_\eb}}.
\end{equation}
Then, formula \eqref{1.minus} reads
\begin{equation}\label{1.mest}
\frac12\frac d{dt}\|\Nx v(t)\|_{L^2_\phi}^2+\alpha\|v(t)\|_{H^3_\phi}^2+(\phi |f(u)|,|u|)\le C_\eb(1+\|g\|^2_{L^2_\phi})+C\eb(\phi |f(u)|,|\Nx v|),
\end{equation}
where the positive constants $C$ and $\alpha$ are independent of $\eb\to0$.
\par
Thus, we only need to estimate the last term in the right-hand side of \eqref{1.mest}. Let us consider the regular and singular cases separately. Let first $f$ be regular and, in addition, assumptions \eqref{1.pol} hold. Then, using the maximal $L^{p+1}$-regularity for the Laplacian in the weighted spaces together with H\"older inequality, we arrive at
\begin{equation}\label{1.lp}
|(f(u),\phi |\Nx v|)|\le C\|f(u)\|_{L^{q}_\phi}\|\Nx v\|_{L^{p+1}_\phi}\le C(\eb^{-1}+\|u\|_{L^{p+1}_\phi}^{p+1}),
\end{equation}
where $\frac1q+\frac1{p+1}=1$ and the constant $C$ is independent of $\eb$. Using again inequalities \eqref{1.pol} and fixing $\eb>0$ small enough, we finally deduce that
\begin{equation}\label{1.m1est}
\frac d{dt}\|\Nx v(t)\|_{L^2_\phi}^2+\alpha\|v(t)\|_{H^3_\phi}^2+\alpha\|u(t)\|_{H^1_\phi}^2+(\phi |f(u)|,|u|)\le C_\eb(1+\|g\|^2_{L^2_\phi}).
\end{equation}
Let us now consider the singular case. The situation here is even simpler since we a priori now that $\|u(t)\|_{L^{\infty}}\le 1$ and, therefore, according to the regularity of the Laplacian in the uniformly local spaces, we conclude that
\begin{equation}
\|\Nx v(t)\|_{L^\infty(\Omega)}\le C.
\end{equation}
By that reason, the second term in the right-hand side of \eqref{1.mest} can be estimated by
$$
|(f(u),\phi|\Nx v|)|\le C\|f(u)\|_{L^1_\phi}
$$
and, since $|f(u)|\le 2f(u).u+C_1$, we again may fix $\eb$ small enough in such a way that the last term in the right-hand side of \eqref{1.mest} will be controlled by the last term in the left-hand side. Thus, in that case, we also arrive at the inequality \eqref{1.m1est}.
\par
After obtaining the differential inequality \eqref{1.m1est}, it is not difficult to deduce the desired estimates  \eqref{1.west} and \eqref{1.wreg} and finish the proof of the theorem. To this end, we note that the Gronwall inequality applied to \eqref{1.mest} gives
\begin{multline}\label{1.phimin}
\|v(T)\|_{H^1_\phi}^2+\int_T^{T+1}\|u(t)\|_{H^1_\phi}^2+(\phi|f(u(t))|,|u(t)|)\,dt\le\\\le C\|v(0)\|^2_{H^1_\phi}e^{-\alpha T}+
C(1+\|g\|^2_{L^2_\phi}).
\end{multline}
In order to deduce estimate \eqref{1.wreg} from \eqref{1.phimin}, it is enough to remind that $\phi=\phi_{\eb,s}$ and take the supremum over $s\in\R$ from the both parts of that inequality (analogous to the derivation of \eqref{1.dis} from \eqref{1.w-dis}).
\par
Let us verify now the smoothing property \eqref{1.wreg}. To this end, we multiply inequality \eqref{1.ddis} by $t$ and integrate by $t\in[0,1]$. That gives,
\begin{equation}\label{1.smm}
t\|u(t)\|_{H^1_\phi}^2+t(F(u(t)),\phi u(t))\le C\int_0^1\|u(s)\|^2_{H^1_\phi}+\|F(u(s))\|_{L^1_\phi}\,ds+C(1+\|g\|_{L^2_\phi}^2).
\end{equation}
Using now inequality \eqref{1.phimin} for estimating the right-hand side of \eqref{1.smm} and using that $F(u)\le f(u).u+C|u|^2$, we arrive at the weighted analogue of estimate \eqref{1.wreg}. Taking finally the supremum over $s\in\R$, we derive the uniformly local estimate \eqref{1.wreg} and finish the proof of the theorem.
\end{proof}
\begin{remark}\label{Rem1.bad} The proof of the last theorem indicates the main difference of the Cahn-Hilliard equation theory in weighted spaces in comparison with the classical unweighted case, namely, the presence of the additional term
$C\eb(|f(u)|,\phi|\Nx v|)$ in the $H^{-1}$-estimate \eqref{1.mest}. This term is not sign-defined and factually destroys the global Lipschitz continuity of the solution semigroup in weighted spaces which is the main technical tool for handling the Cahn-Hilliard equation with singular of fast growing potentials. In the previous theorem, we were able to overcome this difficulty and obtain the dissipative $H^{-1}$-estimate in the weighted and uniformly local spaces in almost the same form as for the non-weighted case. However, as we will see that difficulty leads to  more restrictive assumptions on $f$ if we want to establish the uniqueness.
\end{remark}
\section{Uniqueness and regularity}\label{s2}

In this section, we pose some additional restrictions on the non-linearity $f$ which will allow us to establish the uniqueness of solutions and some crucial smoothing effects for the Cahn-Hilliard equation in the uniformly local spaces. We start with the case of regular potentials and assume
also that the following is true:
\begin{equation}\label{2.freg}
|f'(u)|\le C_1 F(u)+C_2,
\end{equation}
for all $u\in\R$ and for some fixed positive $C_1$ and $C_2$. Note that \eqref{2.freg} does not look as a big restriction for the case of regular potentials (although excludes completely the case of singular $f$ which will be separately discussed below). In particular, the polynomial non-linearities (see condition \eqref{1.pol}) always satisfy this assumption. The next theorem gives the uniqueness of a solution for the Cahn-Hilliard equation with regular potentials.

\begin{theorem}\label{Th2.unireg} Let the assumptions of Theorem \ref{Th1.main} hold and let, in addition, $f$ be regular and \eqref{2.freg} hold. Then, for every $u_0\in\Phi_b$, the solution $(u(t)$ constructed in Theorem \ref{Th1.main} is unique  and, for every two such solutions $u_1(t)$ and $u_2(t)$ (with different initial data) the following weighted Lipschitz continuity in $H^{-1}$ holds:
\begin{equation}\label{2.lip}
\|u_1(t)-u_2(t)\|_{H^{-1}_{\phi_{\eb,s}}}\le Ce^{Kt}\|u_1(0)-u_2(0)\|_{H^{-1}_{\phi_{\eb,s}}},
\end{equation}
where the constants $C$ and $K$ depend on the $\Phi_b$-norms of the initial data $u_1(0)$ and $u_2(0)$, but independent of $\eb\ll1$ and $s\in\R$.
\end{theorem}
\begin{proof} Let $w(t):=u_1(t)-u_2(t)$ and let $v(t):=(-\Dx)^{-1}w(t)$. Then, these functions solve
\begin{equation}\label{2.dif}
\Dt v=\Dx w-l(t)w,
\end{equation}
where $l(t):=\frac{f(u_1(t))-f(u_2(t))}{u_1(t)-u_2(t)}=\int_0^1f'(\kappa u_1(t)+(1-\kappa)u_2(t))\,d\kappa$.
\par
Multiplying, this equation by $\Nx(\phi\Nx v)$ (note that all terms have sense since $u_1$ and $u_2$ are solutions in the sense of Definition \ref{defsol}) and using that $l(t)\ge-K$ and, therefore,
$$
([l(t)+K]w,\phi'\partial_{x_1}v)\le 1/2(\phi[l(t)+K]w,w)+C\eb^2(\phi[l(t)+K]\partial_{x_1}v,\partial_{x_1}v),
$$
we end up with the following estimate (similar to \eqref{1.minus}):
\begin{equation}\label{2.minus}
\frac {d}{dt}\|v\|^2_{H^{1}_{\phi}}+2\alpha\|v\|^2_{H^3_{\phi}}\le K(\|v\|^2_{H^1_{\phi}}+\|v\|^2_{H^2_\phi})+
C\eb^2(\phi l(t)\Nx v,\Nx v),
\end{equation}
where $C$ is independent of $\eb$ being small enough. As usual, the term with the $H^2$-norm of $v$ can be easily estimated by interpolation between
$H^3$ and $H^1$, so, we only need to estimate the last term into the right-hand side of \eqref{2.minus}. Note also that, up to the moment, we have nowhere used the additional restriction that $f$ is regular and satisfies \eqref{2.freg}.
\par
In order
to estimate the term $l(t)$, we remind that the function $F(u)+Ku^2/2$ is convex (since $f'(u)\ge-K$). Consequently, due to condition \eqref{2.freg},
$$
|f'(\kappa u_1+(1-\kappa)u_2)|\le C_1F(\kappa u_1+(1-\kappa)u_2)+C_2\le C_1(F(u_1)+F(u_2))+C_2+K(u_1^2+u_2^2).
$$
Since $F(w)\in L_b^1$ for all $w\in\Phi_b$, we conclude that
\begin{equation}\label{2.f}
\|l(t)\|_{L^1_b(\Omega)}\le C,\ \ t\in\R_+,
\end{equation}
where the constant $C$ is independent of $t$. Thus, estimate \eqref{2.minus} now reads as
\begin{equation}\label{2.mminus}
\frac {d}{dt}\|v\|^2_{H^{1}_{\phi}}+\alpha\|v\|^2_{H^3_{\phi}}\le C\|v\|^2_{H^1_{\phi}}+ C\eb\|\Nx v\|_{L^\infty_{\phi^{1/2}}}^2.
\end{equation}
To estimate the last term in \eqref{2.minus}, it only remains now to use the interpolation inequality
$$
\|\Nx v\|_{L^\infty_{\phi^{1/2}}}^2\le C\|v\|_{H^1_\phi}^{1/2}\|v\|_{H^3_\phi}^{3/2}.
$$
That gives
\begin{equation}\label{2.mmminus}
\frac {d}{dt}\|v\|^2_{H^{1}_{\phi}}+\alpha/2\|v\|^2_{H^3_{\phi}}\le L\|v\|^2_{H^1_{\phi}},
\end{equation}
where the constant $L$ is independent of the shift parameter $s$ in the definition of $\phi=\phi_{\eb,s}$. Applying
the Gronwall inequality to that relation and using the weighted $H^{-1}\to H^1$ regularity of the Laplacian, we end up with \eqref{2.lip} and finish the proof of the theorem.
\end{proof}
\begin{corollary}\label{Cor2.semi} Under the assumptions of Theorem \ref{Th2.unireg}, the Cahn-Hilliard problem \eqref{eqmain} generates a dissipative semigroup $S(t)$ in the phase space $\Phi_b$:
\begin{equation}\label{2.semi}
S(t):\Phi_b\to\Phi_b,\ \ \ S(t)u_0:=u(t),
\end{equation}
where $u(t)$ solves \eqref{eqmain} with $u(0)=u_0$. Moreover, this semigroup is locally Lipschitz continuous in the $H^{-1}_b$-norm:
\begin{equation}\label{2.llip}
\|S(t)u_1-S(t)u_2\|_{H^{-1}_b(\Omega)}\le Ce^{Lt}\|u_1-u_2\|_{H^{-1}_b(\Omega)},
\end{equation}
where constants $C$ and $L$ depend only on the $\Phi_b$-norms of $u_1$ and $u_2$.
\end{corollary}
Indeed, in order to verify \eqref{2.llip}, it is sufficient to take the supremum over the shift parameter $s\in\R$ from both sides of \eqref{2.lip} and the rest assertions are immediate corollaries of Theorem \ref{Th2.unireg}.
\par
We now discuss the further regularity and smoothing property for the solutions of the Cahn-Hilliard equation with regular potentials in the uniformly local spaces.
\begin{corollary}\label{Cor1.smooreg} Let the assumptions of Theorem \ref{Th2.unireg} hold. Then, for every $t>0$, $u(t)\in H^2_b(\Omega)$ and the following estimate is valid:
\begin{multline}\label{2.smooth}
\|u(t)\|_{H^2_b(\Omega)}^2+\|\Dt u\|_{H^{-1}_b(\Omega)}^2+\sup_{\kappa\in\R}\int_t^{t+1}\|\Nx\Dt u(s)\|_{L^2(\Omega_{[\kappa,\kappa+1]})}^2\,ds\le\\\le C\frac{t+1}t\|u_0\|_{\Phi_b}^2e^{-\alpha t}+C(1+\|g\|^2_{L^2_b(\Omega)}),
\end{multline}
where the positive constants $C$ and $\alpha$ are independent of $t$.
\end{corollary}
\begin{proof} Indeed, let us formally differentiate equation \eqref{eqmain} by $t$ (since the solution is unique, this action can be easily justified by the appropriate approximation procedure) and denote $w=\Dt u$ and $v=(-\Dx)^{-1}\Dt u$. Then, these functions solve the analogue of equation \eqref{2.dif}:
\begin{equation}\label{2.dtdif}
\Dt v=\Dx w-f'(u(t))w.
\end{equation}
Multiplying this equation by $\Nx(\phi\Nx v)$ and arguing exactly as in the proof of Theorem \ref{Th2.unireg},
we obtain the estimate \eqref{2.mmminus}. Multiplying that estimate by $t$ and integrating over $[0,t]$, we have
\begin{equation}\label{2.??}
t\|\Nx v(t)\|^2_{L^2_\phi}\le C\int_0^t\|\Nx v(t)\|_{L^2_\phi}^2\,dt,\ \ t\le1.
\end{equation}
Since, according to equation \eqref{eqmain}, $v(t)=\mu(t)$, the right-hand side of \eqref{2.??} can be estimated using \eqref{1.w-dis}. That gives
\begin{equation}\label{1.w-smooth}
t\|\Dt u(t)\|_{H^{-1}_{\phi}}^2\le C(\|u_0\|_{W^{1,2}_\phi}^2+\|F(u_0)\|_{L^1_{\phi}}+\|g\|^2_{L^2_\phi(\Omega)}+1), \ t\le1
\end{equation}
which together with the dissipative estimate \eqref{1.dis} and the trick with taking supremum over the shift parameter, give the desired estimate \eqref{2.smooth} for $\Dt u$. In order to obtain the desired estimate \eqref{2.smooth} for the $H^2_b$-norm, it is now sufficient to use Lemma \ref{Lem1.ellreg}. Corollary \ref{Cor2.semi} is proved.
\end{proof}

Thus, due to the embedding $H^2_b(\Omega)\subset C_b(\Omega)$ the regular non-linearity $f$ becomes subordinated to the linear part of the equation (no matter how fast does it grow) and obtaining the further regularity of solutions is reduced to the standard bootstraping procedure using the highly developed weighted theory for the linear equations. We do not discuss that standard thing here and the rest of the section will be devoted to more interesting and more complicated case of {\it singular} non-linearity $f$. In that case, we have to pose rather restrictive assumption that there exist a {\it convex} function $R:(-1,1)\to\R$ such that
\begin{equation}\label{2.fsing}
\alpha_2R(u)-C_1\le |f(u)|\le \alpha_1 R(u)+C_1,\ \  |f'(u)|\le \alpha_3 |f(u)|^{8/5}+C_3.
\end{equation}
Roughly speaking, condition \eqref{2.fsing} means that the singularities of the function $f$ at $u=\pm1$ are sufficiently strong. In particular,
the function
$$
f(u)=\frac u{(1-u^2)^\gamma}-Ku
$$
satisfies that assumptions if and only if $\gamma\ge 5/3$. Unfortunately, we are unable to handle the most physical case of logarithmic potential
\begin{equation}\label{2.log}
f(u)=\log\frac{1+u}{1-u}-Ku
\end{equation}
and we do not know whether or not the uniqueness result holds for that nonlinearity.
\par
The following theorem gives the analogue of the uniqueness result of Theorem \ref{Th2.unireg} for the case of sufficiently strong singular potentials.
\begin{theorem}\label{Th2.unising} Let the non-linearity $f$ be singular and satisfy assumptions \eqref{fsing} and \eqref{2.fsing}. Let also $g\in L^6_b(\Omega)$. Then, for every $u_0\in\Phi_b$, the solution (in the sense of Definition \ref{defsol}) of equation \eqref{eqmain} is unique and estimate \eqref{2.lip} hold for any two solutions of that equation (with different initial data).
\end{theorem}
\begin{proof}
Let, as in Theorem \ref{Th2.unireg}, $u_1(t)$ and $u_2(t)$ be two solutions of the Cahn-Hilliard equation, $w(t):=u_1(t)-u_2(t)$ and $v(t):=(-\Dx)^{-1}w(t)$. Then, these functions satisfy equation \eqref{2.dif}. In addition, due to assumptions \eqref{2.fsing}, we have
$$
|l(t)|\le C(|f(u_1(t))|^{8/5}+|f(u_2(t))|^{8/5}+1)
$$
which, together with Corollary \ref{Cor1.reg26} give
\begin{equation}\label{2.lreg}
\sup\nolimits_{t\in\R_+,\ s\in\R}\|l\|_{L^{5/4}([t,t+1],L^{15/4}(\Omega_{[s,s+1]}))}\le C=C_{u_1,u_2}.
\end{equation}
Moreover, arguing as in the proof of Theorem \ref{Th2.unireg}, we obtain the differential inequality \eqref{2.minus}.
As we will see below, the regularity \eqref{2.lreg} is sufficient to control the additional weighted term and to close the uniqueness estimate. However, since \eqref{2.lreg} does not imply that $l\in L^{5/4}([t,t+1], L^{15/4}_b(\Omega))$,  we should proceed in a more accurate way. Namely, integrating inequality \eqref{2.minus} over $t$, we end up with
\begin{multline}\label{2.iminus}
\|\Nx v\|_{L^\infty([0,t], L^2_{\phi_\eb})}^2+\alpha\int_0^t\|v(\tau)\|^2_{H^3_{\phi_\eb}}\,d\tau\le\\\le L\int_0^t\|v(\tau)\|_{H^1_{\phi_\eb}}^2\,d\tau+\|\Nx v(0)\|^2_{L^2_{\phi_\eb}}+C\eb^2\int_0^t(|l(\tau)|,\phi_\eb|\Nx v|^2)\,d\tau.
\end{multline}
It is important to note that the positive constants $C$, $L$ and $\alpha$ are independent of $\eb\to0$. In order to estimate the last term in that inequality, we use the following inequalities
\begin{equation}\label{2.wgood}
C^{-1}\int_{\kappa\in\R} \phi_\eb(\kappa)\|v\|_{L^1(\Omega_{[\kappa,\kappa+1]})}\,d\kappa \le\|v\|_{L^1_{\phi_\eb}}\le C\int_{\kappa\in\R} \phi_\eb(\kappa)\|v\|_{L^1(\Omega_{[\kappa,\kappa+1]})}\,d\kappa,
\end{equation}
where the constant $C$ is independent of $\eb\to0$ (see \cite{EZ}). Assuming without loss of generality that $t\le1$ and using the last inequality together with the control \eqref{2.lreg} and H\"older inequality, we have
\begin{multline}\label{2.huge}
\int_0^t(|l(\tau)|,\phi_\eb|\Nx v|^2)\,d\tau\le C\int_{\kappa\in\R}\phi_\eb(\kappa)\|l\cdot|\Nx v|^2\|_{L^1([0,t]\times\Omega_{[\kappa,\kappa+1]})}\,d\kappa\le\\\le C\int_{\kappa\in\R}\phi_\eb(\kappa)
\|l\|_{L^{5/4}([0,t], L^{15/4}(\Omega_{[\kappa,\kappa+1]}))}\cdot\|\Nx v\|^2_{L^{10}([0,t], L^{30/11}( \Omega_{[\kappa,\kappa+1]}))}\,d\kappa\le\\\le C_1\int_{\kappa\in\R}\phi_\eb(\kappa)\|\Nx v\|^2_{L^{10}([0,t], L^{30/11}( \Omega_{[\kappa,\kappa+1]}))}\,d\kappa.
\end{multline}
Using now the following interpolation
$$
\|z\|_{L^{10}([0,t],L^{30/11})}\le C\|z\|_{L^\infty([0,t],L^2)}^{4/5}\|z\|_{L^2([0,t],H^2)}^{1/5}
$$
with the constant $C$ independent of $t\le1$ (we recall that the space dimension $n=3$), we may continue
estimate \eqref{2.huge}:
\begin{multline}\label{2.hhuge}
C\eb^2\int_0^t(|l(\tau)|,\phi_\eb|\Nx v|^2)\,d\tau\le \\\le
 C_2\int_{\kappa\in\R}\phi_\eb(\kappa)\(\eb^{5/2}\|\Nx v\|^2_{L^{\infty}([0,t],L^2(\Omega_{[\kappa,\kappa+1]}))}\)^{4/5}
 \(\|\Nx v\|^2_{L^2([0,t],H^2(\Omega_{[\kappa,\kappa+1]}))}\)^{1/5}\,d\kappa \le \\\le \alpha\int_0^t\|v(\tau)\|^2_{H^3_{\phi_\eb}}\,d\tau+
 C\eb^{5/2}\int_{\kappa\in\R}\phi_\eb(\kappa)\|\Nx v\|_{L^\infty([0,t],L^2(\Omega_{[\kappa,\kappa+1]}))}^2\,d\kappa.
\end{multline}
However, the last term on the right-hand side of \eqref{2.hhuge} still cannot be estimated by the first term on the left-hand side of \eqref{2.iminus} and we have only the obvious one-sided estimate:
\begin{equation}\label{2.wside}
\|\Nx v\|_{L^\infty([0,t],L^2_{\phi_\eb})}^2\le C\int_{\kappa\in\R}\phi_\eb(\kappa)\|\Nx v\|_{L^\infty([0,t],L^2(\Omega_{[\kappa,\kappa+1]}))}^2\,d\kappa.
\end{equation}
In order to overcome this difficulty, we recall that $\phi_{\eb}(x)=\phi_{\eb,s}(x)\ge e^{-\eb}$ if $x\in\Omega_{[s,s+1]}$. Therefore,
$$
\|\Nx v\|_{L^\infty([0,t],L^2(\Omega_{[s,s+1]}))}\le C\|\Nx v\|_{L^\infty([0,t],L^2_{\phi_{\eb,s}})}^2
$$
which together with \eqref{2.hhuge} and \eqref{2.iminus} give the following estimate
\begin{multline}\label{2.local}
\|\Nx v\|_{L^\infty([0,t],L^2(\Omega_{[s,s+1]}))}^2\le C\(\|\Nx v(0)\|^2_{L^2_{\phi_{\eb,s}}}+\int_0^t\|\Nx v(\tau)\|^2_{H^1_{\phi_{\eb,s}}}\,d\tau\)+\\+ C\eb^{5/2}\int_{\kappa\in\R}\phi_{\eb,s}(\kappa)\|\Nx v\|_{L^\infty([0,t],L^2(\Omega_{[\kappa,\kappa+1]}))}^2\,d\kappa,
\end{multline}
where $t\le1$ and the constant $C$ is independent of $s\in\R$ and $\eb\to0$.
\par
We are now ready to close the uniqueness estimate. To this end, we multiply estimate \eqref{2.local} by $\phi_{\eb/2,l}(s)$,
where $l\in\R$ is a new shift parameter and integrate over $s\in\R$. Then, using the obvious estimate
\begin{equation}\label{2.ok}
\int_{s\in\R}\phi_{\eb/2,l}(s)\int_{\kappa\in\R}\phi_{\eb,s}(\kappa)z(\kappa)\,d\kappa\,ds\le C\eb^{-1}\int_{\kappa\in\R}\phi_{\eb/2,l}(\kappa)z(\kappa)\,d\kappa,
\end{equation}
see \cite{EZ} for details, we arrive at
\begin{multline}\label{2.llocal}
\int_{\kappa\in\R}\phi_{\eb/2,s}(\kappa)\|\Nx v\|_{L^\infty([0,t],L^2(\Omega_{[\kappa,\kappa+1]}))}^2\,d\kappa\le\\\le C\eb^{-1}\(\|\Nx v(0)\|^2_{L^2_{\phi_{\eb/2,s}}}+\int_0^t\|\Nx v(\tau)\|^2_{H^1_{\phi_{\eb/2,s}}}\,d\tau\)+\\+ C\eb^{3/2}\int_{\kappa\in\R}\phi_{\eb/2,s}(\kappa)\|\Nx v\|_{L^\infty([0,t],L^2(\Omega_{[\kappa,\kappa+1]}))}^2\,d\kappa.
\end{multline}
Fixing here $\eb>0$ to be small enough (say, $C\eb^{3/2}= 1/2$) and using \eqref{2.wside}, we conclude that
$$
\|\Nx v(t)\|^2_{L^2_{\phi_{\eb/2,s}}}\le C\|\Nx v(0)\|^2_{L^2_{\phi_{\eb/2,s}}}+C\int_0^t\|\Nx v(\tau)\|^2_{L^2_{\phi_{\eb,s}}}\,d\tau.
$$
This estimate together with the Gronwall inequality give the desired estimate \eqref{2.lip} and finish the proof of the theorem.
\end{proof}
The following corollary is a straightforward analog of Corollaries \ref{Cor2.semi} and \ref{Cor1.smooreg} for the singular case.
\begin{corollary}\label{Cor2.sing} Let the assumptions of Theorem \ref{Th2.unising} hold. Then, problem \eqref{eqmain} generates a dissipative semigroup \eqref{2.semi} in the phase space $\Phi_b$ which is locally Lipschitz continuous in it (i.e., \eqref{2.llip} holds). Moreover, any solution of \eqref{eqmain} possesses the smoothing property \eqref{2.smooth}.
\end{corollary}
Indeed, these assertions follow from Theorem \ref{Th2.unising} exactly as in the regular case, so we omit their proofs here.
\par
Let us mention that, in contrast to the regular case, the proved $H^2_b$-regularity \eqref{2.smooth} of solutions is not sufficient to apply the further regularity using the linear theory since we still have the singular term $f(u)$ and need to prove that the solution cannot reach the singular points $u=\pm1$. The next proposition gives such a result.

\begin{proposition}\label{Prop2.sep} Let the assumptions of Theorem \ref{Th2.unising} hold. Then, $f(u(t))\in L^\infty(\Omega)$ for all $t>0$ and
\begin{equation}\label{2.finfty}
\|f(u(t))\|_{L^\infty(\Omega)}\le \frac{1+t^N}{t^N}e^{-\alpha t}Q(\|u_0\|_{\Phi_b})+Q(\|g\|_{L^6_b(\Omega)}),
\end{equation}
where $\alpha$ and $N$ are some positive constants and $Q$ is a monotone function which are independent of $t$ and $u_0$.
\end{proposition}
\begin{proof} The proof of this result is analogous to \cite{MZ} and the fact that the underlying domain $\Omega$ is unbounded does not make any difference. Nevertheless, for the convenience of the reader, we give below a schematic derivation of the desired estimate. First, according to \eqref{2.smooth} and the embedding $H^1\subset L^6$,
$$
\|\Dt(-\Dx)^{-1}u(t)\|_{L^6_b(\Omega)}\le \frac{1+t}{t}e^{-\alpha t}Q(\|u_0\|_{\Phi_b})+Q(\|g\|_{L^6_b(\Omega)})
$$
and, therefore, analogously to Corollary \ref{Cor1.reg26},
$$
\|f(u(t))\|_{L^6_b(\Omega)}+\|u(t)\|_{W^{2,6}_b(\Omega)}+\|\Nx u(t)\|_{L^\infty(\Omega)}\le \frac{1+t}{t}e^{-\alpha t}Q(\|u_0\|_{\Phi_b})+Q(\|g\|_{L^6_b(\Omega)}),
$$
where we have implicitly used the maximal regularity theorem for the Laplacian in $L^6_b$ and the embedding $W^{1,6}\subset L^\infty$. Finally, using this estimate together with assumptions \ref{2.fsing}, we get
\begin{multline}\label{2.good}
\|f(u(t))\|_{W^{1,15/4}_b(\Omega)}\le C\|\Nx u\|_{L^\infty(\Omega)}\|f'(u(t))\|_{L^{15/4}_b(\Omega)}\le\\\le C\|\Nx u\|_{L^\infty(\Omega)}(1+\|f(u(t))\|_{L^6_b(\Omega)})^{8/5}\le \frac{1+t^{13/5}}{t^{13/5}}e^{-\alpha t}Q(\|u_0\|_{\Phi_b})+Q(\|g\|_{L^6_b(\Omega)})
\end{multline}
and the desired estimate \eqref{2.finfty} follows now from the embedding $W^{1,15/4}\subset L^\infty$.
\end{proof}

\begin{remark}\label{Rem2.sep} Since $\lim_{u\pm1} f(u)=\infty$, there exists a strictly positive (monotone decreasing) function $\delta_f$ depending only on $f$ such that
\begin{equation}\label{2.ssep}
\|u(t)\|_{L^\infty(\Omega)}\le 1-\delta_f(\|f(u(t))\|_{L^\infty(\Omega)}).
\end{equation}
Thus, estimate \eqref{2.good} shows that any solution $u(t)$ of problem \eqref{eqmain} is indeed
separated (uniformly in time) from the singularities $u=\pm1$. This allows to obtain the further regularity of the solution exactly as in the regular case.
\end{remark}

\section{Attractors}\label{s3}
This section is devoted to the long-time behavior of solutions of the Cahn-Hilliard problem. We first discuss the relatively simple case where the uniqueness holds and after that consider the more delicate situation where we do not have the uniqueness result. We also recall that, as usual for the case of unbounded domains (see \cite{MZ} and references therein), the corresponding attractor is not compact in the (uniform) topology of the initial phase space and does not attract in that space, so one should consider the so-called {\it locally}-compact attractors  which attract bounded sets of the initial phase space in the appropriate {\it local} topology. For the convenience of the reader, we recall the definition of such an attractor adapted to the case of Cahn-Hilliard equations.
\begin{definition}\label{Def3.attr} Let the assumptions of Theorems \ref{Th2.unireg} or \ref{Th2.unising} hold and let $S(t):\Phi_b\to\Phi_b$ be the solution semigroup associated with the Cahn-Hilliard equation \eqref{eqmain}. A set $\mathcal A$ is a locally-compact attractor of the Cahn-Hilliard equation if the following conditions are satisfied:
\par
1) $\mathcal A$ is bounded in $\Phi_b$ and is compact in $\Phi_{loc}:=W^{1,2}_{loc}(\overline{\Omega})$;
\par
2) it is strictly invariant: $S(t)\mathcal A=\mathcal A$ for $t\ge0$;
\par
3) it attracts  bounded sets of $\Phi_b$ in the topology of $\Phi_{loc}$, i.e., for every $B$ bounded in $\Phi_b$ and every neighborhood $\Cal O(\Cal B)$ of the set $\Cal A$, there exists time $T=T(B,\Cal O)$ such that
$$
S(t)B\subset\Cal O(\Cal A) \ \ \text{if}\ \ t\ge T.
$$
\end{definition}
Recall that the compactness in $\Phi_{loc}$ simply means that the restriction $\Cal A\big|_{\Omega_{[S_1,S_2]}}$ to any bounded subcylinder $\Omega_{[S_1,S_2]}$ is compact in $W^{1,2}(\Omega_{[S_1,S_2]})$ and the attraction property means that
$$
\lim_{t\to\infty}\dist_{W^{1,2}(\Omega_{[S_1,S_2]})}(S(t)B\big|_{\Omega_{[S_1,S_2]}},\Cal A\big|_{\Omega_{[S_1,S_2]}})=0
$$
for any bounded set $B$ of $\Phi_b$ and any bounded subcyliner $\Omega_{[S_1,S_2]}$. Here and below $\dist_V$ denotes the Hausdorff semidistance in $V$.
\begin{theorem}\label{Th3.glattr} Let the assumptions of Theorem \ref{Th2.unireg} or \ref{Th2.unising} hold. Then, the Cahn-Hilliard equation possesses a locally-compact attractor $\Cal A$ in the sense of the above definition which is generated by all bounded complete solutions of that equation
\begin{equation}\label{3.str}
\Cal A=\Cal K\big|_{t=0},
\end{equation}
where $\Cal K=\{u\in L^\infty(\R,\Phi_b),\ \ u \text{ solves  \eqref{eqmain}}\}$.
 Moreover, this attractor is bounded in $W^{2,2}_b(\Omega)$ and, in the singular case, the attractor is separated from the singularities:
 \begin{equation}\label{3.sep}
 \|u_0\|_{C_b(\Omega)}\le 1-\delta,\ \ \forall u_0\in\Cal A,
\end{equation}
 for some positive $\delta$.
\end{theorem}
\begin{proof} Indeed, thanks to the abstract theorem on the attractor existence, we need to construct a bounded in $\Phi_b$ and (pre)compact in $\Phi_{loc}$ absorbing set $\Cal B$ of the semigroup $S(t)$ and to verify that $S(t)$ restricted to $\Cal B$ has a closed graph (see \cite{BV1}).
\par
Assume first that the equation is regular and assumptions of Theorem \ref{Th2.unireg} hold. Then, due to Corollary \ref{Cor1.smooreg}, the $H^2_b$-ball
$$
\Cal B:=\{u\in H^2_b(\Omega),\ \|u\|_{H^2_b}\le R\}
$$
is an absorbing set for the semigroup $S(t)$ (if $R$ is large enough) which is obviously compact in $\Phi_{loc}$. The fact that the graph of $S(t)$ is closed on $\Cal B$ is an immediate corollary of the Lipschitz continuity \eqref{2.lip} in a weaker topology.
\par
Let now the equation be singular and the assumptions of Theorem \ref{Th2.unising} hold. Then, due to Corollary \ref{Cor2.sing}, Proposition \ref{Prop2.sep} and estimate \eqref{2.ssep}, the set
$$
\Cal B:=\{ u\in H^2_b(\Omega),\ \ \|u\|_{H^2_b}\le R,\ \ \|u\|_{C_b(\Omega)}\le 1-\delta\}
$$
will be an absorbing set for the semigroup $S(t)$ (if $R$ is large enough and $\delta>0$ is small enough) and it is again compact in $\Phi_{loc}$. The fact that the graph of $S(t)$ is closed is again an immediate corollary of the Lipschitz continuity \eqref{2.lip} proved in Theorem \ref{Th2.unising}.
\par
Thus, the assumptions of the abstract attractors existence theorem are verified in both case. The description \eqref{3.str} is also a standard corollary of that theorem and the fact that attractor is bounded in $H^2_b$ and, in the singular case, is separated from singularities is also immediate since the attractor is a subset of any absorbing set. Theorem \ref{Th3.glattr} is proved.
\end{proof}
\begin{remark} The structure and regularity of the obtained attractor may be further investigated in a standard way. In particular, since the solution on the attractor is proved to be separated from singularities and globally bounded, the usual bootstraping arguments show that the factual regularity of the attractor is restricted only by the smoothness of the domain $\Omega$, nonlinearity $f$ and the external forces $g$ (if all of them are of $C^\infty$ the attractor will also belong to $C^\infty$). Moreover, arguing in a standard way, one may show that this attractor will typically have infinite Hausdorff dimension (it will be so, e.g., if there exists at least one spatially-homogeneous exponentially unstable equilibrium), may obtain the upper and lower bounds for it's Kolmogorov's $\varepsilon$-entropy and so on (see \cite{MZ} for the detailed discussion of a general scheme).
\end{remark}
We now turn to discuss the case where the uniqueness theorem does not hold and only the assumptions of Theorem \eqref{Th1.main}
are satisfied, in particular, it will be so for the case of logarithmic potential \eqref{2.log}. Since, we only have the existence (but not uniqueness) of a solution, we will use the so-called trajectory approach, see \cite{CV,MZ} for more details.
\par
As a first step, following the general scheme,  we introduce the so-called trajectory dynamical system associated with the Cahn-Hilliard equation \eqref{eqmain}.

\begin{definition}\label{Def3.tr} Let a set $K^+\subset L^\infty(\R_+,\Phi_b)$ be the set of all weak solutions $u:\R_+\to\Phi_b$
of the Cahn-Hilliard problem \eqref{eqmain} in the sense of Definition \ref{defsol} (since the $\mu$-component of the solution is uniquely determined by its $u$-component, we omit the $\mu$-part of the solution in that definition) such that the assumptions \eqref{eqdef} are satisfied and the following analogue of the main dissipative estimate holds:
\begin{multline}\label{3.dis}
\|u(t)\|_{W^{1,2}_b(\Omega)}^2+\|F(u(t))\|_{L^1_b(\Omega)}+\|\Dx u\|_{L^2_b([t,t+1]\times\Omega)}^2+\\+\|f(u)\|^2_{L^2_b([t,t+1]\times\Omega)}+\|\Nx\mu\|^2_{L^2_b([t,t+1]\times\Omega)}\le C_ue^{-\alpha t}+C(\|g\|^2_{L^2_b(\Omega)}+1),
\end{multline}
where the constant $C$ is the same as in estimate \eqref{1.dis}, $t\ge0$ is arbitrary (almost arbitrary, being pedantic) and $C_u$ is some positive number depending on the trajectory $u$.
\par
It is not difficult to verify that the translation semigroup
\begin{equation}\label{3.trsem}
T(h):\, K_+\to K_+,\ \ (T(h)u)(t):=u(t+h)
\end{equation}
acts on on the set $K_+$ and, in the case with uniqueness, this semigroup is conjugated to the standard solution semigroup $S(h)$  acting on the usual phase space $\Phi_b$. By this reason, the set $K_+$ is called the  trajectory phase space of problem \eqref{eqmain} and the translation semigroup $T(h)$ acting on this space is often referred as a trajectory dynamical system associated with this equation.
\end{definition}
We intend to find an attractor for the introduced trajectory dynamical system (=trajectory attractor for the initial Cahn-Hilliard equation). To this end, as usual, we need to specify the class of "bounded" sets in $K_+$ and fix the appropriate topology in $K_+$.
\begin{definition}\label{Def3.bound} A set of trajectories  $B\subset K_+$ is bounded if the dissipative estimate \eqref{3.dis}
holds uniformly with respect to all $u\in B$ with the same constant $C_B$, i.e.,
$$
C_u\le C_B<\infty,\ \ \forall u\in B.
$$
Comparing  the dissipative estimates \eqref{1.dis} and \eqref{3.dis}, we see that, at least in the case of uniqueness, the constant $C_u$ is simply related with the $\Phi_b$-norm of the initial data $u_0$. Therefore, in the case of uniqueness, the class of bounded sets thus defined corresponds to the usual bounded sets in the phase space $B$ and gives a natural extension of that concept to the case without uniqueness.
\end{definition}
In this section, we will consider only the so-called weak (trajectory) attractors, so we introduce the topology in $K_+$ in the following way.
\begin{definition}\label{Def2.top} We endow the trajectory phase space $K_+$ by the {\it weak-star} topology of the space
$$
\Theta_+:=L^\infty_{loc}(\R_+,W^{1,2}_{loc}(\overline{\Omega}))\cap L^2_{loc}(\R_+,W^{2,2}_{loc}(\overline{\Omega})).
$$
We recall that $u_n\to u$ in that topology iff, for every time segment $[t,t+T]$ and every finite cylinder $\Omega_{[-S,S]}$,
$u_n\to u$ weak-star in $L^\infty([t,t+T],W^{1,2}(\Omega_{[-S,S]}))$ and weakly in $L^2([t,t+T], W^{2,2}(\Omega_{[-S,S]}))$.
\par
It is important for the attractor theory that every bounded set of $\Theta_+$ is precompact and  metrizable in this weak-star topology (see \cite{CV,RR} for the details).
\end{definition}
We are now ready to define attractor of the trajectory semigroup $T(h)$ acting on $K_+$.
\begin{definition}\label{Def3.trattr} A set $\Cal A_{tr}\subset K_+$ is a global attractor of the trajectory dynamical system
$(T(h),K_+)$ (=trajectory attractor of the Cahn-Hilliard equation \eqref{eqmain}) if
\par
1) $\Cal A_{tr}$ is compact in $K_+$ (endowed by the weak-star topology of $K_+$);
\par
2) strictly invariant: $T(h)\Cal A_{tr}=\Cal A_{tr}$;
\par
3) attracts all bounded (in the sense of Definition \ref{Def3.bound}) sets of $K_+$ in the weak-star topology of $\Theta_+$.
\end{definition}
Finally, the next theorem gives the existence of such an attractor.
\begin{theorem}\label{Th3.trattr} Let the assumptions of Theorem \ref{Th1.main} holds. Then, the Cahn-Hilliard equation \eqref{eqmain} possesses a trajectory attractor $\Cal A_{tr}$ in the sense of the above definition. Moreover, this attractor is also generated by all bounded complete solutions of problem \eqref{eqmain}
\begin{equation}\label{3.strtr}
\Cal A_{tr}=\Cal K\big|_{t\ge0},
\end{equation}
where $\Cal K\subset L^\infty(\R,\Phi_b)$ is a set of all solutions $u:\R\to\Phi_b$ which are defined for all $t$ and satisfy the dissipative estimate \eqref{3.dis} for all $t\in\R$ with $C_u=0$.
\end{theorem}
\begin{proof} We first note that the set $K_+$ is not empty due to Theorem \ref{Th1.main}, so the trajectory dynamical system is reasonably defined. Next, analyzing the dissipative estimate \eqref{3.dis} and the definition of a bounded set, we see that the
set $\Cal B\subset K_+$ of all trajectories $u$ which satisfy this estimate with the constant $C_u\le1$ is a bounded absorbing set of the semigroup $T(h)$ and even
$$
T(h)\Cal B\subset\Cal B
$$
for all $h\in\R_+$. Moreover, $\Cal B$ is precompact and metrizable in the topology of $K_+$ (weak-star topology of $\Theta_+$) and clearly the shift semigroup $T(h)$ is continuous in that topology. Thus, in order to be able to apply the standard attractor existence theorem, we only need to verify that $\Cal B$ is closed as a subset of $\Theta_+$ with the weak-star topology. Since this proof is  standard and repeats word by word the proof of the existence of a weak solution (which is constructed exactly by the weak-star limit of solutions of the appropriate approximate problem), we rest  it for the reader.
\par
Thus, due to the abstract attractor existence theorem (see \cite{BV1,CV}), the trajectory attractor exists and is generated by all bounded complete solutions of the Cahn-Hilliard problem \eqref{eqmain}. Theorem \ref{Th3.trattr} is proved.
\end{proof}
\begin{remark}\label{Rem3.str} Using the fact that the appropriate norm of the time derivative $\partial_t u$ is under the control for every weak solution $u$ (due to the dissipative estimate and the first equation \eqref{eqmain}), one can verify that the trajectory attractor $\Cal A_{tr}$ attracts bounded sets of $K_+$ in a {\it strong} topology of
$$
\Theta_+(\nu):=L^\infty_{loc}(\R_+,W^{1-\nu,2}_{loc}(\Omega))\cap L^2_{loc}(\R_+,W^{2-\nu,2}_{loc}(\Omega)),
$$
for every $\nu>0$.
\end{remark}

\section{Weighted energy equalities}\label{s4}
In this section we will mainly consider the case of {\it singular} potentials without uniqueness. We first check that any
weak solution of the Cahn-Hilliard satisfies the weighted  energy {\it equalities} and then, in the next section, prove that the weak trajectory attractor constructed before is compact in a strong topology and the attraction holds in the strong topology as well.
\par
We start with the following lemma which is the key part of our proof of the weighted energy equalities.
\begin{lemma}\label{Lem4.ae} Let the function $u:\R_+\times\Omega\to\R$ be such that
\begin{multline}\label{4.as}
u\in L^\infty(\R_+,H^1_b(\Omega))\cap L^2_b(\R_+,H^2_b(\Omega)),\\ \Dt u\in L^2_b(\R_+,H^{-1}_b(\Omega)),\ \ F(u)\in L^\infty(\R_+,L^1(\Omega))
\end{multline}
and
$$
H:=\Dx u-f(u)\in L^2_b(\R_+,H^{1}_b(\Omega)),
$$
where the nonlinearity $f$ satisfies assumptions \eqref{fsing}.
Then, for all weight functions $\varphi\in L^1(\R)$ satisfying \eqref{1.phi} and  almost all $T_1,T_2\in\R_+$, $T_2>T_1$, we have
\begin{multline}\label{4.energy}
\frac12[(\varphi,|\Nx u(T_2)|^2)-(\varphi,|\Nx u(T_1)|^2)]+[(\varphi, F(u(T_2)))-(\varphi, F(u(T_1)))]=\\=\int_{T_1}^{T_2}(\varphi'\partial_{x_1}u(t),\Dt u(t))-(H(t),\varphi\Dt u(t))\,dt.
\end{multline}
\end{lemma}
\begin{remark}\label{Rem.strange} The main difficulty in the proof of this and the next lemmas is that we do not have the maximal $H^1\to H^3$-regularity for the semilinear heat equation
$$
\Dx u-f(u)=H.
$$
By this reason, we are unable to deduce that $\Dx u$ and $f(u)$ {\it separately} belong to $H^1_b$. Thus, although the inner product $(\varphi\Dt u,\Dx u-f(u))$ is well-posed, the terms $(\varphi\Dx u,\Dt u)$ and $(\varphi\Dt u, f(u))$ can be nevertheless {\it ill-posed} and we cannot use the standard methods to verify the energy equality. Instead of that, we obtain the result (following \cite{MiZe}) using the trick based on the convexity arguments. An alternative method, based on the abstract energy equality for the maximal monotone operators (see \cite{Bre}), can be found in \cite{RS} Lemma 4.1 (the analogous result is proved there for the case of bounded domains).
\end{remark}
\begin{proof}[Proof of the lemma] We first note that, without loss of generality, we may think that the potential $F$ is {\it convex} (actually, in a general situation it differs from the convex one by the non-essential linear term). Then, following \cite{MiZe}, we may write out the following inequalities which hold for all $h>0$:
\begin{multline}\label{a.good}
\frac {F(u(\tau+h))-F(u(\tau))}h=f(u(\tau))\frac{u(\tau+h)-u(\tau)}h+\\+\int_0^1r_2f'(u(\tau)+r_1r_2[u(\tau+h)-
u(\tau)]) \dd r_1 \dd r_2\frac{(u(\tau+h)-u(\tau))^2}h\ge\\\ge f(
u(\tau))\frac{u(\tau+h)-u(\tau)}h
\end{multline}
and
\begin{multline}\label{a.verygood}
\frac {F(u(\tau+h))-F(u(\tau))}h=f(u(\tau+h))\frac{u(\tau+h)-
u(\tau)}h-\\\int_0^1(1-r_2)f'( u(\tau+h)-r_1(1-r_2)[ u(\tau+h)-
u(\tau)]) \dd r_1 \dd r_2\frac{(u(\tau+h)-u(\tau))^2}h\\\le f(
u(\tau+h))\frac{u(\tau+h)- u(\tau)}h.
\end{multline}
Multiplying these inequalities by $\varphi$ and integrating over $x$, we end up with
\begin{multline}
\(f(u(\tau)),\varphi\frac{u(\tau+h)-u(\tau)}h\)\le
\(\varphi,\frac {F(u(\tau+h))-F(u(\tau))}h\)\le\\\le \(f(
u(\tau+h)),\varphi\frac{u(\tau+h)- u(\tau)}h\)
\end{multline}
(actually, due to our assumptions on $u$, all terms in that inequality are well-defined for almost all $\tau$).
In addition, since $(\varphi,|\Nx u(\tau)|^2)$ is also a {\it convex} functional, we have the analogous, but simpler inequalities
\begin{multline}
\(\Nx u(\tau),\varphi\frac{\Nx u(\tau+h)-\Nx u(\tau)}h\)\le
\frac12\(\varphi,\frac {|\Nx u(\tau+h)|^2-|\Nx u(\tau)|^2}h\)\le\\\le \(\Nx
u(\tau+h),\varphi\frac{\Nx u(\tau+h)- \Nx u(\tau)}h\).
\end{multline}
Taking a sum of these two inequalities, integrating by parts and using the definition of $H(\tau)$, we arrive at
\begin{multline}\label{4.large}
\(-H(\tau),\varphi\frac{\Nx u(\tau+h)-\Nx u(\tau)}h\)+\(\partial_{x_1} u(\tau),\varphi\frac{u(\tau+h)-u(\tau)}h\)\le\\\
\frac{\mathcal E_\varphi(u(\tau+h))-\mathcal E_\varphi(u(\tau))}h\le\\\
\(-H(\tau+h),\varphi\frac{\Nx u(\tau+h)-\Nx u(\tau)}h\)+\(\partial_{x_1} u(\tau+h),\varphi\frac{u(\tau+h)-u(\tau)}h\)
\end{multline}
with
$$
\mathcal E_\varphi(u):=\frac12(\varphi,|\Nx u|^2)+(F(u),1).
$$
Finally, integrating this formula over $\tau\in[T_1,T_2]$, we have
\begin{multline}\label{4.last}
\int_{T_1}^{T_2}\(-H(\tau),\varphi\frac{\Nx u(\tau+h)-\Nx u(\tau)}h\)+\(\partial_{x_1} u(\tau),\varphi\frac{u(\tau+h)-u(\tau)}h\)\,d\tau\le\\\
\frac1h \int_{T_2}^{T_2+h}\mathcal E_\varphi(u(\tau))\,d\tau-\frac1h\int_{T_1}^{T_1+h}\mathcal E_\varphi(u(\tau))\,d\tau\le\\\
\int_{T_1}^{T_2}\(-H(\tau+h),\varphi\frac{\Nx u(\tau+h)-\Nx u(\tau)}h\)+\(\partial_{x_1} u(\tau+h),\varphi\frac{u(\tau+h)-u(\tau)}h\)\,d\tau.
\end{multline}
It only remains to note that, due to our assumptions on $u$, we may pass to the limit $h\to0$ for almost all fixed $T_1$ and $T_2$. This gives the desired equality \eqref{4.energy}.
\end{proof}
As the next step, we need \eqref{4.energy} to hold for {\it every} $T_1,T_2\in\R_+$. That is proved in the following lemma.

 \begin{lemma}\label{Lem4.energy} Let the assumptions of Lemma \ref{Lem4.ae} hold and let, in addition, the function $u\in C([0,T],L^2_\varphi(\Omega))$ and the function $\tau\to\mathcal E_\varphi(u(\tau))$ be lover semicontinuous, i.e.,
 $$
 \mathcal E_\varphi(u(\tau))\le\liminf_{n\to\infty}\mathcal E_\varphi(u(\tau_n))
 $$
 for any $\tau_n\to\tau$ and any $\tau$. Then, $\tau\to\mathcal E_\varphi(u(\tau))$ is absolutely continuous and \eqref{4.energy} holds for all $T_1>0$ and $T_2>0$.
\end{lemma}
\begin{proof} We first note that, due to the lower semicontinuity,
\begin{equation}\label{4.leq}
\Cal E_\varphi(u(T_2))-\Cal E_\varphi(u(T_1))\le \int_{T_1}^{T_2}(\partial_{x_1} u(t),\phi'\Dt u(t))-(H(t),\Dt u(t))\,dt
\end{equation}
for {\it all} $T_2$ and {\it almost all} $T_1$. Assume now that we proved that this {\it inequality} holds {\it for all} $T_1$ as well. Then the assertion of the lemma holds. Indeed, let we have the {\it strict} inequality for some $T_1>0$ and $T_2>0$. Then, we may find $T_1^*<T_1$ and $T_2^*>T_2$ such that the equality holds on the interval $[T_1^*,T_2^*]$ (since it holds for almost all $T_1$ and $T_2$). Splitting the interval $[T_1^*,T_2^*]=[T_1^*,T_1]\cup[T_1,T_2]\cup[T_2,T_2^*]$ and using \eqref{4.leq} for first and third interval together with the strict inequality on the second interval, we see that the inequality must be strict also on the interval $[T_1^*,T_2^*]$. Thus, in order to prove the lemma, we only need to verify that inequality \eqref{4.leq} holds for every $T_1$. Without loss of generality, we may prove that for $T_1=0$ only.
\par
To this end, we note that the function $u$ is a {\it unique} solution of the following Cahn-Hilliard type problem
\begin{equation}\label{4.mod}
\varphi^{-1/2}(-\Dx)^{-1}(\varphi^{1/2}\Dt v)=\varphi^{-1}\Nx(\varphi \Nx v)-f(v)-\tilde g(t),\ \ v\big|_{t=0}=u\big|_{t=0}
\end{equation}
with $\tilde g(t):=H(t)-\varphi^{-1/2}(-\Dx)^{-1}(\varphi^{1/2}\Dt u)+\varphi^{-1}\varphi'\partial_{x_1}u$. Indeed, by the construction $u$ is a weak solutions of that equation. Let $v_1(t)$ and $v_2(t)$. Then, writing out the equation for the difference $w=v_1-v_2$, multiplying this equation by $\varphi w$ and arguing in a standard way (exploiting the monotonicity of $f$), we arrive at
$$
\|w(t)\|_{H^{-1}_\varphi}\le Ce^{Kt}\|w(0)\|_{H^{-1}_\varphi}
$$
and the uniqueness holds.
\par
Note also that $H\in L^2_b(\R_+,H^1_b(\Omega))$ by the assumptions on $u$. So, we only need to verify the energy {\it inequality} for the auxiliary equation \eqref{4.mod}. To this end, we approximate the singular potential $f$ by the {\it regular} ones $f_n$ just by replacing $f(u)$ outside of $(-1+1/n,1-1/n)$, by the proper linear function:
\begin{equation}\label{4.modmod}
\varphi^{-1/2}(-\Dx)^{-1}(\varphi^{1/2}\Dt v_n)=\varphi^{-1}\Nx(\varphi \Nx v_n)-f_n(v_n)-\tilde g(t),\ \ v_n\big|_{t=0}=u\big|_{t=0}.
\end{equation}
Then, since \eqref{4.modmod} is a small compact perturbation of the standard Cahn-Hilliard equation in $\Omega$ with {\it linearly} growing nonlinearity $f_n$, we have the unique solvability as well as the energy equality:
$$
\mathcal E_{\varphi,n}(v_n(T))-\mathcal E_{\varphi,n}(u(0))+\int_0^T\|\varphi^{1/2}\Dt v_n(t)\|^2_{H^{-1}}\,dt=\int_0^T(\varphi \tilde g(t),\Dt v_n(t))\,dt.
$$
where $\mathcal E_{\varphi,n}(z)$ is the energy with the potential $F$ replaced by $F_n$.
Now, passing to the limit $n\to\infty$, $v_n\to u$ and  using the obvious relations
$$
\mathcal E_{\varphi,n}(u(0))\to\mathcal E_\varphi(u(0));\ \ \mathcal E_\varphi(u(T))\le \liminf_{n\to\infty}\mathcal E_{\varphi,n}(v_n(T))
$$
together with the weak convergence $\Dt v_n\to\Dt u$ in $L^2(0,T,H^{-1}_\varphi(\Omega))$ (these results are standard for the
Cahn-Hilliard equations theory, so we do not present the  proofs here, see  \cite{MZ1} for more details), we arrive at
$$
\mathcal E_{\varphi}(u(T))-\mathcal E_{\varphi}(u(0))+\int_0^T\|\varphi^{1/2}\Dt u(t)\|^2_{H^{-1}}\,dt=
\int_0^T(\varphi \tilde g(t),\Dt u(t))\,dt
$$
which is equivalent to the desired energy inequality \eqref{4.leq} and that finishes the proof of the lemma.
\end{proof}
\begin{remark}\label{Rem4.better} Being pedantic, we have proved in Lemma \ref{Lem4.energy} that the energy equality holds for $T_1>0$ only. This drawback can be easily corrected by the proper extending the function $u(t)$ to negative times and arguing as before. Thus, \eqref{4.energy} remains true for $T_1=0$ as well.
\end{remark}
We are now able to return to the initial Cahn-Hilliard problem.
\begin{corollary}\label{Cor4.energy} Let $u$ be a weak solution of the Cahn-Hilliard problem \eqref{eqmain}. Then, for any weight function $\varphi\in L^1$ satisfying \eqref{1.phi}, the function $t\to\mathcal E_\varphi(u(t))$ is absolutely continuous on $[0,\infty)$ and the following energy identity holds:
\begin{equation}\label{4.chenergy}
\frac d{dt}\(\mathcal E_\varphi(u(t))+(u\varphi,g)\)+(\varphi,|\Nx\mu(t)|^2)+(\Dt u(t),\varphi'\partial_{x_1}u(t))+(\partial_{x_1}\mu(t),\varphi'\mu(t))=0
\end{equation}
for almost all $t$.
\end{corollary}
Indeed, in order to obtain \eqref{4.chenergy}, it is sufficient to multiply the initial equation \eqref{0.ch} by $\varphi\mu$
and use the energy equality proved in Lemma \ref{Lem4.energy} (all the assumptions of the lemma are clearly satisfied for any weak solution).
\begin{corollary}\label{Cor4.cont} Let $u$ be a weak solution of the Cahn-Hilliard problem \eqref{eqmain} in the case of singular potentential
$f$. Then $u\in C([0,T],\Phi_\varphi)$ for any $T>0$ and any integrable weight $\varphi$ satisfying \eqref{1.phi}. The latter means that $u\in C([0,T], H^1_\varphi(\Omega))$ and $F(u)\in C([0,T],L^1_\varphi(\Omega))$.
\end{corollary}
\begin{proof} Without loss of generality, we prove the continuity at $t=0$ only. Indeed, let $t_n\to0$. Then, from the energy equality, we see that $\Cal E_\varphi(u(t_n))\to\Cal E_\varphi(u(0))$. Then, from the convexity arguments and using that $u\in C([0,T],L^2_\varphi)$, we see that
$$
\|u(t_n)\|_{L^2_\varphi}\to\|u(0)\|_{L^2_\varphi},\ \ \ \|F(u(t_n))\|_{L^1_\varphi}\to\|F(u(0))\|_{L^1_\varphi}.
$$
The first convergence together with the obvious weak convergence $u(t_n)\to u(0)$ in $H^1_\varphi$ gives the strong convergence in that space. And the second convergence together with the convergence $F(u(t_n))\to F(u(0))$ almost everywhere and the standard fact from the Lebesgue integration theory (let $z_n\ge0$, $z_n\to z$ almost everywhere and $\int z_n\to\int z$. Then $z_n\to z$ in $L^1$) imply the convergence $F(u(t_n))\to F(u(0))$ in $L^1_\varphi(\Omega)$ and the corollary is proved.
\end{proof}
We conclude this section by stating one more standard fact which is however important for what follows.
\begin{proposition}\label{Prop4.stupid} Let the function $u\in H^2_b(\Omega)\cap\{u\big|_{\partial\Omega}=0\}$ be such that
$f(u)\in L^2_b(\Omega)$ and let $F_{1/2}(u):=\int_0^u\sqrt{f'_0(v)}\,dv$ (see \eqref{f0}) and $\varphi$ be the integrable weight satisfying \eqref{1.phi}. Then, $F_{1/2}(u)\in H^1_b(\Omega)$ and
\begin{equation}\label{4.12}
(\Nx(\varphi\Nx u),f_0(u))=-(\varphi,|\Nx F_{1/2}(u)|^2).
\end{equation}
\end{proposition}
\begin{proof} Indeed, approximating the singular function $f_0(u)$ by the regular ones $f_n(u)$ (just replacing $f'_0(u)$ by the appropriate constant when $u$ is outside of $(-1+1/n,1-1/n)$, we see that, clearly, $f_n(u)\to f_0(u)$ almost everywhere and in $L^2_\varphi(\Omega)$. On the other hand, since $f_n(u)$ is now regular $C^1$-function,
$$
(\Nx(\varphi\Nx u),f_n(u))=-(\varphi,f'_n(u)|\Nx u|^2).
$$
Passing to the limit here and using that $f'_n$ is monotone increasing, we conclude that
$f'_0(u)|\Nx u|^2\in L^1_\varphi(\Omega)$ and
$$
(\Nx(\varphi\Nx u),f_0(u))=-(\varphi,f'_0(u)|\Nx u|^2).
$$
Thus, we only need to prove that $\Nx F_{1/2}(u)=\sqrt{f'_0(u)}\Nx u$ in the sense of distributions, but that can be easily verified, e.g., by truncating  the function $u$ with an appropriate constant outside of $(-1+1/n,1-1/n)$ and passing to the limit $n\to\infty$). Thus, Proposition \ref{Prop4.stupid} is proved.
\end{proof}
\begin{corollary}\label{Cor4.stupid} Any weak energy solution $u$ of the Cahn-Hilliard problem \eqref{eqmain} satisfies
$$
F_{1/2}(u)\in L^2([0,T],H^1_\varphi(\Omega)),
$$
for every $T$ and every integrable weight $\varphi$.
\end{corollary}
Indeed, we have the control of the $L^2_\varphi$-norm of the solution $u$ from the energy estimate. Thus, the assertion is an immediate corollary of the previous proposition.

\section{Strong attraction via the energy method}\label{s5}
In this concluding section, we develop the weighted energy method and improve essentially the results on the trajectory attractors obtained in Section \ref{s3}. We start with simplifying the construction of the trajectory dynamical system (see Definitions \ref{Def3.tr} and \ref{Def3.bound}).

\begin{proposition}\label{Prop5.tr} Under the assumptions of Theorem \ref{Th3.trattr}, any weak solution $u$ of the Cahn-Hilliard problem (see Definition \ref{defsol}) satisfies estimates \eqref{1.dis} and \eqref{3.dis} with $C_u:=C\|u(0)\|_{\Phi_b}^2$. Thus, condition \eqref{3.dis} in the Definition \ref{Def3.tr} is automatically satisfied (and can be omitted). Moreover, the class of bounded sets introduced in Definition \ref{Def3.bound} possesses the following natural description:
$$
B\subset K_+ \  \text{ is bounded if and only if }\ \ B\big|_{t=0}:=\{u(0),\ u\in B\}\ \ \text{ is bounded in }\Phi_b.
$$
\end{proposition}
Indeed, since the weighted energy equality holds for every weak solution $u$, repeating word by word the derivation of \eqref{1.dis}, we derive estimate \eqref{3.dis} for {\it any} weak solution $u$ with $C_u:=C\|u(0)\|^2_{\Phi_b}$. The other assertions of the proposition follow immediately from this estimate.
\par
We now introduce the natural {\it strong} topology on the space $K_+$.

\begin{definition}\label{Def5.strong} Let $F_{1/2}(u):=\int_0^u \sqrt {f'_0(v)}\,dv$ and let
\begin{multline}
\Theta^+_{strong}:=\{u\in C_{loc}(\R_+,H^1_{loc}(\overline\Omega)\cap L^2_{loc}(\R_+,H^2_{loc}(\overline\Omega))\cap H^1_{loc}(\R_+,H^{-1}_{loc}(\overline\Omega)),\\ \mu\in L^2_{loc}(\R_+,H^1_{loc}(\overline\Omega)),\ \ F(u)\in C_{loc}(\R_+,L^1_{loc}(\overline\Omega)),\\ f(u)\in L^2_{loc}(\R_+\times\overline\Omega),\ \
F_{1/2}(u)\in L^2_{loc}(\R_+,H^1_{loc}(\overline\Omega))\}.
\end{multline}
Then the topology induced by the embedding $K_+\subset \Theta^+_{strong}$ is called a {\it strong} topology on the trajectory phase space $K_+$.
\end{definition}

\begin{remark} The above definition is not self-contradictory, since any weak solution $u$ belongs to $\Theta^+_{strong}$ and $K_+$ is  a subset of that space, see Corollaries \ref{Cor4.cont} and \ref{Cor4.stupid}.
\end{remark}

We are now able to state the main result of this section.

\begin{theorem}\label{Th5.strong} Let the assumptions of Theorem \ref{Th3.trattr} hold and let $f(u)$ be a singular potential. Then, the trajectory dynamical system $(T(h),K_+)$ associated with the Cahn-Hilliard equation possesses the compact global attractor in the strong topology of $\Theta^+_{strong}$. Moreover, this attractor coincides with the weak trajectory attractor constructed in Theorem \ref{Th3.trattr}.
\end{theorem}
\begin{proof} Indeed, let $h_n\to\infty$, $u_n\in K_+$ from a bounded set (so $\|u_n(0)\|_{\Phi_b}\le C$) and $v_n:=T(h_n)u_n$. Then, according to Theorem \ref{Th3.trattr}, without loss of generality we may assume that $v_n\to u$ in a a {\it weak} topology of $\Theta_+$. And, to prove the theorem, we only need to verify that $v_n\to u$ in a {\it strong} topology.
\par
Let $\varphi$ be an integrable weight satisfying \eqref{1.phi} and let us rewrite the energy equality \eqref{4.chenergy} for $v_n$ as follows:
\begin{multline}\label{5.energyn}
[\Cal E_\varphi(v_n(T))+(v_n(T)\varphi,g)]-[\Cal E_{\varphi}(u_n(0))+(u_n(0)\varphi,g)]e^{-\alpha (T+h_n)}+\\+\int_{-h_n}^T e^{-\alpha(T-t)} \Cal H_\varphi(v_n(t))\,dt=0
\end{multline}
with
\begin{multline}
\Cal H_\varphi(v_n(t)):=(\varphi, |\Nx \mu_n(t)|^2)+(\Dt v_n(t),\varphi'\partial_{x_1} v_n(t))+\\+(\partial_{x_1}\mu_n(t),\varphi'\mu_n(t))-\alpha[\Cal E_\varphi(v_n(t))+(v_n(t)\varphi,g)]
\end{multline}
and with sufficiently small $\alpha>0$ which will be specified below and any $T>0$. And, of course, for the limit function $u(t)$, we have the following energy equality:
\begin{equation}\label{5.energylim}
[\Cal E_\varphi(u(T))+(u(T)\varphi,g)]+\int_{-\infty}^T e^{-\alpha(T-t)} \Cal H_\varphi(u(t))\,dt=0.
\end{equation}
As usual for the energy method, we need to pass to the weak limit $n\to\infty$ in \eqref{5.energyn} and compare the result with the limit equation \eqref{5.energylim}. Indeed, for the first term, we have as before,
$$
\Cal E_\varphi(u(T))\le\liminf_{n\to\infty}\Cal E_\varphi(v_n(T))
$$
and $(v_n(T)\varphi,g)\to(u(T)\varphi,g)$.
The second term obviously tends to zero (since $u_n(0)$ are bounded and $h_n\to\infty$), so, we only need to establish the analogous inequality for the third term. To this end, we transform it as follows:
\begin{multline}\label{5.huge}
\Cal H_\varphi(v_n)=\frac14(\varphi,|\Nx\mu_n(t)|^2)+[\frac14(\varphi,|\Nx\mu_n(t)|^2)-2\beta(\varphi,|\mu_n(t)|^2)]+\\+
\beta [(\varphi,|\Dx v_n(t)|^2)+(\varphi,|f(v_n)|^2)+4(\varphi,|\Nx F_{1/2}(v_n(t))|^2)]+\\+[4\beta(\varphi'\partial_{x_1}v_n(t),f(v_n(t)))-4\beta K(\varphi,|\Nx v_n(t)|^2)-\alpha(v_n(t)\varphi,g)]+\\+[(\varphi,\beta |f(v_n(t))|^2-\alpha F(v_n(t)))]+\\+
[\beta(\varphi,|\Dx v_n(t)|^2)+\frac12(\varphi,|\Nx \mu_n(t)|^2)-\frac12\alpha(\varphi,|\Nx v_n(t)|^2)+\\+ (\partial_{x_1}\mu_n(t),\varphi'\mu_n(t))+((\Dx)^{-1} \mu_n(t),\varphi'\partial_{x_1} v_n(t))]=\\=I_1(v_n)+I_2(v_n)+I_3(v_n)+I_4(v_n)+I_5(v_n)+I_6(v_n),
\end{multline}
where $\beta>0$ is a sufficiently small positive number. Let us pass to the limit in every term of \eqref{5.huge} separately. Indeed, since, without loss of generality,  $\mu_n\to\mu$ weakly in $L^2([S,S+1],H^1_\varphi)$, for all $S$,
$$
\int_{-\infty}^Te^{-\alpha(T-t)}I_1(u(t))\,dt\le\liminf_{n\to\infty}\int_{h_n}^Te^{-\alpha(T-t)}I_1(v_n(t))\,dt.
$$
The analogous estimate for the second term follows from the fact that (due to the Friedrichs inequality), this term is a positive definite quadratic form. The estimate for the third term follows from the weak convergences $\Dx v_n\to\Dx u$,
$f(v_n)\to f(u)$ and $\Nx F_{1/2}(v_n)\to\Nx F_{1/2}(u)$ in the space $L^2([S,S+1], L^2_\varphi(\Omega))$.
\par
The fourth term is trivial since we have the {\it strong} convergence $\Nx v_n\to\Nx u$ in $L^2([S,S+1],L^2_\varphi(\Omega))$.
The desired estimate for the 5th term follows from the convergence $v_n\to u$ almost everywhere, the estimate $\beta f^2(z)-\alpha F(z)\ge -C$ and the Fatou lemma. Finally, the 6th term is a quadratic form with respect to $v_n$ and $\mu_n$
which will be also positive definite if $\eb>0$ (from \eqref{1.phi}) and $\alpha$ are small enough. Thus, the desired estimate for $I_6$ is also true and we arrive at
$$
\int_{-\infty}^Te^{-\alpha(T-t)}\Cal H_\varphi(v_n(t))\,dt\le \liminf_{n\to\infty}\int_{h_n}^Te^{-\alpha(T-t)}\Cal H_\varphi(u(t))\,dt.
$$
Comparing this estimate with \eqref{5.energylim}, we see that it is possible only if all of the above inequalities are,
in a fact,  {\it equalities}. Thus, we have factually verified that
\begin{equation}\label{5.conv}
\begin{aligned}
\|\mu_n\|_{L^2([S,S+1],H^1_\varphi)}\to\|\mu\|_{L^2([S,S+1],H^1_\varphi)},\\
\|\Dx v_n\|_{L^2([S,S+1],L^2_\varphi)}\to\|\Dx u\|_{L^2([S,S+1],L^2_\varphi)},\\
\|\Nx F_{1/2}(v_n)\|_{L^2([S,S+1],L^2_\varphi)}\to\|\Nx F_{1/2}(u)\|_{L^2([S,S+1],L^2_\varphi)},\\
\|f(v_n)\|_{L^2([S,S+1],L^2_\varphi)}\to\|f(u)\|_{L^2([S,S+1],L^2_\varphi)},\\
\|\Nx v_n(T)\|_{L^2_\varphi}\to\|\Nx u(T)\|_{L^2_\varphi},\\
\|F(v_n(T)\|_{L^1_\varphi}\to\|F(u(T))\|_{L^1_\varphi}
\end{aligned}
\end{equation}
for all fixed $S$ and $T$. That, together with the weak convergences implies the desired {\it strong} convergence for
$\mu_n$, $\Dx v_n$, $\Nx F_{1/2}(v_n)$ and $f(v_n)$. Thus, we only need to verify that
$$
v_n\to u\ \  \text{in }C([S,S+1], H^1_\varphi)\ \ \text{and}\ \ F(v_n)\to F(u) \ \text{in}\ C([S,S+1],L^1_\varphi).
$$
We note that, although the last two convergences of \eqref{5.conv} imply that $v_n(T)\to u(T)$ in $\Phi_\varphi$, that does not give straightforwardly the desired {\it uniform} convergence in $C([S,S+1],\Phi_\varphi)$. However, these convergences imply in a standard way that the {\it trace}  $\Cal A:=\Cal A_{tr}|_{t=0}$ of the trajectory attractor $\Cal A_{tr}$ to the phase space is {\it compact} in $\Phi_\varphi$ and that
\begin{equation}\label{5.comp}
\sup_{s\in[S,S+1]}d_{\Phi_\varphi}(v_n(s),\Cal A)\to 0
\end{equation}
as $n\to\infty$. Let us show that it is sufficient to verify the desired uniform convergences. Indeed, let us fix the orthonormal basis ${e_n}$ in the space $H^1_\varphi$ and the corresponding orhtoprojector $P_N$ of the first $N$ veectors and $Q_N:=1-P_N$. Then, on the one hand,
$$
P_N v_n\to P_N u
$$
in $C([S,S+1],H^1_\varphi)$ for any $N$ (since we have from the very beginning the strong convergence in $C([S,S+1],L^2_\varphi)$. On the other hand, \eqref{5.comp} together with the compactness of the attractor $\Cal A$ in $H^1_\varphi$ imply that, for any $\eb>0$, we may find $N=N(\eb)$ such that
$$
\limsup_{n\to\infty}\|Q_N v_n\|_{C([S,S+1],H^1_\varphi)}\le \eb.
$$
Thus, the uniform convergence in $C([S,S+1],H^1_\varphi)$ is verified.
\par
Let us now prove that $F(v_n)\to F(u)$ in $C([S,S+1],L^1_\varphi)$. To this end, we note that, due to the absolute continuity of the Lebesgue integration, the compactness of $F(\Cal A)$ in $L^1_\varphi$ and convergence \eqref{5.comp} imply that, for every $\eb>0$, there is $\delta=\delta(\eb)>0$ such that
$$
\liminf_{n\to\infty}\sup_{s\in[S,S+1]}\int_A F(v_n(s))\,dx\le \eb
$$
if $\meas_\varphi(A):=\int_A\varphi\,dx<\delta$. The desired uniform convergence $F(v_n)\to F(u)$ in $C([S,S+1],L^1_\varphi)$ is now an immediate corollary of the Egorov theorem. Thus, all desired strong convergences are verified and the theorem is proved.
 \end{proof}

\end{document}